\documentclass{mcom-l}

\usepackage{amsmath,amssymb,amsxtra,comment,graphicx,psfrag}
\usepackage{bm,mathrsfs}
\usepackage{mathtools}
\usepackage{xspace}
\usepackage{color}
\usepackage{enumerate}
\usepackage{hyperref}
\usepackage{pdfsync}
\usepackage{graphicx}
\usepackage{fancyhdr}

\allowdisplaybreaks

\definecolor{blue}{rgb}{0,0.0,0.9}

 at 8. truept

\def\d{{\rm d}}
\def\Id{{\rm Id}}

\def\ii{\textrm{i}}

\newtheorem{theorem}{Theorem}[section]
\newtheorem{lemma}{Lemma}[section]
\newtheorem{proposition}{Proposition}[section]

\newtheorem{remark}[proposition]{Remark}

\numberwithin{equation}{section}
\numberwithin{table}{section}
\numberwithin{figure}{section}

\begin{document}

\title[]
{Sharp convergence rates of time discretization for stochastic time-fractional PDEs subject to additive space-time white noise } 

\author[]{Max Gunzburger}
\address{Department of Scientific Computing, Florida State
University, Tallahassee, FL 32306, USA.} 
\email {\href{mailto:gunzburg@fsu.edu}{gunzburg{\it @}fsu.edu}}

\author[]{Buyang Li}
\address{Department of Applied Mathematics, 
The Hong Kong Polytechnic University, Hung Hom, Hong Kong.} 
\email {\href{mailto:buyang.li@polyu.edu.hk}{buyang.li{\it @}polyu.edu.hk}}

\author[]{Jilu Wang}
\address{Department of Scientific Computing, Florida State University, Tallahassee, FL 32306, USA.}
\curraddr{Department of Mathematics and Statistics, 
Mississippi State University, Mississippi State, 39762, USA.}
\email {\href{mailto:jwang13@fsu.edu}{jwang13{\it @}fsu.edu}} 
\date{}




\begin{abstract}
The stochastic time-fractional equation 
$\partial_t \psi -\Delta\partial_t^{1-\alpha} \psi = f + \dot W$ 
with space-time white noise $\dot W$ is discretized in time by a backward-Euler convolution quadrature  
for which the sharp-order error estimate 
\[
({\mathbb E}\|\psi(\cdot,t_n)-\psi_n\|_{L^2(\mathcal{O})}^2)^{\frac{1}{2}}=O(\tau^{\frac{1}{2}-\frac{\alpha d}{4}})
\]
is established for $\alpha\in(0,2/d)$, where $d$ denotes the spatial dimension, $\psi_n$ the approximate solution at the $n^{\rm th}$ time step, and $\mathbb{E}$ the expectation operator. 
In particular, the result indicates sharp convergence rates of numerical solutions for 
both stochastic subdiffusion and diffusion-wave problems in one spatial dimension. 
Numerical examples are presented to illustrate the theoretical analysis. \\[5pt]
\end{abstract}

\maketitle

\thispagestyle{headings}
\markright{\small Mathematics of Computation (accepted)\hfill }

\section{Introduction}\label{sec:intro}
We are interested in the convergence of numerical methods for solving the stochastic time-fractional partial differential equation (PDE) problem
\begin{align}\label{Frac-SPDE} 
\left\{\begin{aligned}
&\partial_t \psi(x,t)-\Delta \partial_t^{1-\alpha}\psi(x,t) =   
f(x,t)+\dot W(x,t)    
&& (x,t)\in \mathcal{O}\times {\mathbb R}_+ \\
&\psi(x,t)=0 && (x,t)\in \partial\mathcal{O}\times {\mathbb R}_+  \\
&\psi(x,0)=\psi_0(x) && x\in \mathcal{O},
\end{aligned}\right.
\end{align} 
where $\mathcal{O}\subset\mathbb{R}^d$, $d\in\{1,2,3\}$, denotes a bounded region with Lipschitz boundary $\partial\mathcal{O}$, $f(x,t)$ a given deterministic source function, $\psi_0(x)$ given deterministic initial condition, and $\dot W(x,t)$ a space-time white noise, i.e., the time derivative of a cylindrical Wiener process in $L^2(\mathcal{O})$. The underlying probability sample space for the stochastic noise is denoted by $\Omega$. The operator $\Delta: D(\Delta)\rightarrow L^2(\mathcal{O})$ 
denotes the Laplacian, defined on the domain 
$$
D(\Delta)=\{\phi\in H^1_0(\mathcal{O}): \Delta \phi\in L^2(\mathcal{O})\} ,
$$ 
and $\partial_t^{1-\alpha}\psi$ denotes the left-sided Caputo fractional time derivative of order $1-\alpha\in(-1,1)$, defined by (c.f. \cite[pp. 91]{KST})
\begin{align} \label{Caputo}
\partial_t^{1-\alpha} \psi(x,t) 
:= \left\{
\begin{aligned}
&\frac{1}{\Gamma(\alpha)}  \int_0^t (t-s)^{\alpha-1} \frac{\partial\psi(x,s)}{\partial s} \d s 
&&\mbox{if}\,\,\, \alpha\in(0,1], \\
&\frac{1}{\Gamma(\alpha-1)} 
\int_0^t (t-s)^{\alpha-2}\psi(x,s)\d s 
&&\mbox{if}\,\,\, \alpha\in(1,2), 
\end{aligned}
\right.
\end{align}
where $\Gamma(s):=\int_0^\infty t^{s-1}e^{-t}\d t$ denotes Euler's gamma function. 

Problem  \eqref{Frac-SPDE} arises, e.g, when considering heat transfer in a material with thermal memory based on a modified Fick's law \cite{Choi-MacCamy-1989,Gurtin-PipKin-1968,MacCamy1977,Nunziato1971}, subject to stochastic noise \cite{ClementDaPrato,KP,MijenaNane}. 
For the model \eqref{Frac-SPDE}, both the fractional time derivative and the stochastic process forcing result in solution having low regularity. Hence, the numerical approximation of such problems and the corresponding numerical analysis are very challenging. 
By defining $\partial_t^{\alpha} \psi(x,t):=\partial_t^{\alpha-1}\partial_t \psi(x,t)$ for $\alpha\in(1,2)$ and using the identity 
\begin{align}
\partial_t^{\alpha-1}\partial_t^{1-\alpha}\psi(x,t)
=
\left\{
\begin{aligned}
&\psi(x,t)-\psi(x,0) &&\mbox{if}\,\,\,\alpha\in(0,1),\\
&\psi(x,t) &&\mbox{if}\,\,\,\alpha\in(1,2),
\end{aligned}
\right. 
\end{align}
applying $\partial_t^{\alpha-1}$ to \eqref{Frac-SPDE} yields another formulation of \eqref{Frac-SPDE}: 
\begin{align}
\partial_t^\alpha \psi(x,t) - \Delta\psi(x,t) 
=
\left\{
\begin{aligned}
&\partial_t^{\alpha-1}( f(x,t)+\dot W(x,t) ) 
-\Delta\psi(x,0) &&\mbox{if}\,\,\,\alpha\in(0,1),\\
& f(x,t)+\dot W(x,t) 
-\Delta\psi(x,0) &&\mbox{if}\,\,\,\alpha=1,\\
&\partial_t^{\alpha-1}(f(x,t) +\dot W(x,t) )  &&\mbox{if}\,\,\,\alpha\in(1,2) ,
\end{aligned}
\right. 
\end{align}
where the case $\alpha=1$ can be verified directly from \eqref{Frac-SPDE}. For the sake of clarity, we focus on only one of the equivalent problems, namely 
\eqref{Frac-SPDE}.

The solution of \eqref{Frac-SPDE} can be decomposed into the solution of the deterministic problem 
\begin{align}\label{Deter-SPDE2} 
\left\{
\begin{array}{ll}
\partial_t v(x,t) -\Delta \partial_t^{1-\alpha}v(x,t) = f (x,t)   &  (x,t)\in \mathcal{O}\times {\mathbb R}_+   \\
v(x,t)=0  & (x,t)\in \partial\mathcal{O}\times {\mathbb R}_+     \\
v(x,0)=\psi_0(x)   & x\in \mathcal{O}     
\end{array}\right.
\end{align}
plus the solution of the stochastic problem
\begin{align}\label{Frac-SPDE2} 
\left\{\begin{aligned}
&\partial_t u(x,t) -\Delta \partial_t^{1-\alpha} u(x,t) = \dot W(x,t) 
&& (x,t)\in \mathcal{O}\times {\mathbb R}_+ \\
&u(x,t)=0   &&(x,t)\in \partial\mathcal{O}\times {\mathbb R}_+ \\ 
&u(x,0)=0   && x\in \mathcal{O}  .
\end{aligned}\right.
\end{align} 
The stability and convergence of numerical solutions of \eqref{Deter-SPDE2} have been widely studied \cite{CCP2007,CuestaLubichPalencia:2006,LubichSloanThomee:1996,McLeanMustapha:2015,MustaphaSchotzau2014}. For example, if $f$ is smooth in time then numerical methods of up to order $6$ are available for approximating the solution of \eqref{Deter-SPDE2} 
and its equivalent formulations \cite{JinLazarovZhouSISC2016,JinLiZhou-BDF,JinLiZhou-CN,JLZ-MaxLp,LLSWZ2018,LWZ2017,LubichSloanThomee:1996,McLeanMustapha:2015}. In particular, the convolution quadrature generated by the backward Euler method yields a first-order convergence rate for solving \eqref{Deter-SPDE2}. 

In this work, we focus on numerical approximation of the stochastic time-fractional PDE \eqref{Frac-SPDE2} with additive space-time white noise based on the convolution quadrature generated by the backward Euler method. In the case $\alpha\in(1,2)$ and $d=1$, rigorous error estimates for numerical solutions of \eqref{Frac-SPDE2} are carried out in \cite{KP} for the case of additive Gaussian noise in the general $Q$-Wiener process setting. For a space-time white noise, an almost optimal-order convergence rate for the time discretization error  
\begin{align}\label{Estimate-epsilon}
({\mathbb E}\|u(\cdot,t_n)-u_n\|_{L^2(\mathcal{O})}^2)^{\frac{1}{2}}=O(\tau^{\frac{1}{2}-\frac{\alpha }{4}-\epsilon})
\end{align} 
is proved \cite[Remark 4.7, with $\rho=\alpha$]{KP}) for arbitrarily small $\epsilon>0$, where $u(\cdot,t_n)$ and $u_n$ denote the mild solution and numerical solution of \eqref{Frac-SPDE2} at time $t_n$, respectively. The estimate \eqref{Estimate-epsilon} is ``almost optimal'' in the sense that the optimal approximation theoretic error estimate for functions having the regularity of the solution $u$ does not have the $\epsilon$ term in the exponent.  
We are not aware of any rigorous numerical analyses in the case $\alpha\in(0,1)$. 
In the case $\alpha=1$, error estimates for time discretization of the stochastic PDE \eqref{Frac-SPDE2} are proved in \cite{GyongyNualart1997} and \cite{AllenNovoselZhang1998,DuZhang2002,Gyongy1999,Yan} for Rothe's method and the backward Euler method, respectively, with different spatial discretization methods. In particular, the convergence rate $O(\tau^{\frac14})$ in time was proved, which corresponds to $\epsilon=0$ in \eqref{Estimate-epsilon}. Some modern references on numerical analysis for stochastsic PDEs in the case $\alpha=1$ include \cite{GWZ-2014,Jentzen-Kloeden-2016,Lord-Powell-Shardlow-2014,Zhang-Karniadakis-2017}.  

The error estimate \eqref{Estimate-epsilon} is consistent with the H\"older continuity of the solution $u\in C^{\gamma}([0,T];L^2(\Omega;L^2(\mathcal{O})))$ and the pathwise $\gamma$-H\"older continuity, with arbitrary $\gamma\in(0,\frac12-\frac{\alpha d}{4})$; see Appendix \ref{Append} or \cite[Corollary 1]{MijenaNane}.
The aim of this paper is to prove, for general $d$-dimensional domains, the sharper order convergence estimate 
\begin{align}\label{Estimate-e}
({\mathbb E}\|u(\cdot,t_n)-u_n\|_{L^2(\mathcal{O})}^2)^{\frac{1}{2}}=O(\tau^{\frac{1}{2}-\frac{\alpha d}{4}})\qquad \mbox{ $\alpha\in(0,2/d)$, \, $d\in\{1,2,3\}$} 
\end{align}
for time discretization of the stochastic PDE \eqref{Frac-SPDE2}, where by ``sharp'' we mean that we are able to obtain an approximation theoretic convergence rate that is consistent with respect to the regularity of the solution in time. This estimate is achieved via a more delicate analysis of the resolvent operator by using its Laplace transform representation. Our result covers both subdiffusion and diffusion-wave cases in one-dimensional spatial domains and, for the subdiffusion case, multi-dimensional domains. 

The rest of the paper is organized as follows. In Section \ref{The main results}, we present the backward-Euler convolution quadrature scheme we use to determine approximate solutions of the stochastic time-fractional PDE \eqref{Frac-SPDE2} and then state our main theoretical results. In Section \ref{Proof of Theorem}, we derive an integral representation of the numerical solution for which we prove sharp convergence rate results for the approximate solution. Numerical results are given in Section \ref{Numerical examples} to illustrate the theoretical analyses. 

Throughout this paper, we denote by $C$, with/without a subscript, a generic constant independent of $n$ and $\tau$ which could be different at different occurrences.

\section{The main results}
\label{The main results}
In this section, we describe the time discretization scheme we use for determining approximate solutions of the stochastic time-fractional PDE \eqref{Frac-SPDE} and state our main results about the convergence rate of the numerical solutions.  

\subsection{Mild solution of the stochastic PDE} 

Let $\phi_j(x)$, $j=1,2,\dots$, denote the $L^2$-norm normalized eigenfunctions of the Laplacian operator $-\Delta$ corresponding to the eigenvalues $\lambda_j$, $j=1,2,\dots,$ arranged in nondecreasing order. 
The cylindrical Wiener process on $L^2(\mathcal{O})$ can be represented as 
(cf. \cite[Proposition 4.7, with $Q=I$ and $U_1$ denoting some negative-order Sobolev space]{PratoZabczyk2014})
\begin{align}\label{white-noise} 
W(x,t) =\sum_{j=1}^\infty \phi_j(x) W_j(t)  
\end{align} 
with independent one-dimensional Wiener processes $W_j(t)$, $j=1,2,\dots$. 

In the case $\psi_0=0$, the solution of the deterministic problem \eqref{Deter-SPDE2} can be expressed by (via Laplace transform, cf. \cite[(3.11) and line 4 of page 12]{LubichSloanThomee:1996}) 
\begin{align}\label{Deter-sol-repr}
v(\cdot,t) = \int_0^t  E(t-s) f(\cdot,s)\d s ,
\end{align}
where the operator $E(t):L^2(\mathcal{O})\rightarrow L^2(\mathcal{O})$ is given by 
\begin{equation}\label{eqn:EF}
E(t) \phi:=\frac{1}{2\pi {\rm i}}\int_{\Gamma_{\theta,\kappa}}e^{zt} z^{\alpha-1} (z^\alpha-\Delta )^{-1}\phi\, \d z  \quad
\forall\, \phi\in L^2(\mathcal{O}) ,
\end{equation}
with integration over a contour $\Gamma_{\theta,\kappa}$ on the complex plane, 
\begin{align}\label{contour-Gamma}
  \Gamma_{\theta,\kappa} 
  &\, =\left\{z\in \mathbb{C}: |z|=\kappa  , |\arg z|\le \theta\right\}\cup
  \{z\in \mathbb{C}: z=\rho e^{\pm {\rm i}\theta}, \rho\ge \kappa \}  \nonumber\\
  &=:
\Gamma_{\theta,\kappa}^\kappa+\Gamma_{\theta,\kappa}^\theta . 
\end{align}
The angle $\theta$ above can be any angle such that $\pi/2<\theta<\min(\pi,\pi/\alpha)$ 
so that, for all $z$ to the right of $\Gamma_{\theta,\kappa}$ in the complex plane,  
$z^\alpha\in\Sigma_{\alpha\theta}:=\{z\in\mathbb{C}\backslash \{0\}:|\arg z|\le\alpha\theta \}$ with $\alpha\theta<\pi$.

Correspondingly, the mild solution of \eqref{Frac-SPDE2} is define as (cf. \cite{MijenaNane} and \cite[Proposition 2.7]{KP})  
\begin{align}
u(\cdot,t)&= 
\int_0^t  E(t-s)\d W(\cdot,s) \label{Mild-sol-} \\
&
= \sum_{j=1}^\infty \int_0^t E(t-s)\phi_j\d W_j(s)  . \label{Mild-sol}
\end{align}
This mild solution is well defined in $C^{\gamma}([0,T];L^2(\Omega;L^2(\mathcal{O})))$ for arbitrary $\gamma\in(0,\frac12-\frac{\alpha d}{4})$; see Appendix \ref{Append}.

\subsection{Convolution quadrature} 

Let $\{t_n=n\tau\}_{n=0}^N$ denote a uniform partition of the interval $[0,T]$ with a time step size $\tau = T/N$, and let $u^n=u(x,t_n)$. Under the zero initial condition, the Caputo fractional time derivative $\partial_t^{1-\alpha} u(x,t_n)$ can be discretized by the backward-Euler convolution quadrature \cite{Lubich:1986} (also known as Gr\"unwald-Letnikov approximation, cf. \cite{CuestaLubichPalencia:2006}) 
\begin{align}
\bar\partial_\tau^{1-\alpha} u_n
=\frac{1}{\tau^{1-\alpha} } \sum_{j=0}^n b_{n-j}u_j ,\quad n=0,1,2,\dots,N, 
\end{align} 
where $b_j$, $j=0,1,2,\dots N$, are the coefficients in the power series expansion 
\begin{align}
(1-\zeta)^{1-\alpha}=\sum_{j=0}^\infty b_{j}\zeta^j  .
\end{align} 
Here, $1-\zeta$ is the characteristic function of the backward-Euler method and 
we set 
\begin{align}\label{definition-d}
\delta_\tau(\zeta)= \frac{1-\zeta}{\tau}  \quad \textrm{for }\zeta\in {\mathbb C}\backslash[1,\infty) .
\end{align}

For any given sequence $\{v_n\}_{n=0}^\infty\in \ell^2(L^2(\mathcal{O}))$, we denote 
\begin{align}\label{gener-func}
\widetilde v(\zeta)=\sum_{n=0}^\infty v_n\zeta^n \quad \textrm{for }\zeta\in {\mathbb D}
\end{align} 
which is referred to as the generating function of the sequence $\{v_n\}_{n=0}^\infty$ (see \cite{LubichSloanThomee:1996}). 
Clearly, $\widetilde v$ is an $L^2(\mathcal{O})$-valued analytic function in the unit disk ${\mathbb D}$ and the limit 
$$
\widetilde v(e^{i\theta})=\lim_{r\rightarrow 1_-}\widetilde v(re^{i\theta})
$$ 
exists in $L^2(0,2\pi;L^2(\Omega))$. 
Then, we have
\begin{equation}\label{generate-dv}
\begin{aligned}
\sum_{n=0}^\infty (\bar\partial_\tau^{1-\alpha} v_n)\zeta^n
&=\sum_{n=0}^\infty  \frac{1}{\tau^{1-\alpha}} \sum_{j=0}^n b_{n-j}v_j\zeta^n
\\&=(\delta_\tau(\zeta))^{1-\alpha}\sum_{j=0}^\infty v_j\zeta^j 
=(\delta_\tau(\zeta))^{1-\alpha}\widetilde v(\zeta) .
\end{aligned} 
\end{equation}

\subsection{Time-stepping scheme and main theorem}

With the notations introduced in the last subsection, we discretize the fractional-order derivative $\partial_t^{1-\alpha}$ in \eqref{Frac-SPDE2} by using convolution quadrature in time to obtain  
\begin{equation}\label{CQ-scheme2}
\frac{u_n -u_{n-1}}{\tau} -\Delta \bar\partial_\tau^{1-\alpha} u_n 
= \frac{W(\cdot,t_n)-W(\cdot,t_{n-1})}{\tau} .
\end{equation}
Equivalently, $u_n$ can be expressed as
\begin{equation}\label{CQ-scheme-}
\begin{aligned}
u_n
=& \big(\Id- \tau^{\alpha}  b_0   \Delta \big)^{-1} 
u_{n-1}+  \tau^{\alpha} \sum_{j=0}^{n-1} b_{n-j}\Delta\big(\Id- \tau^{\alpha}  b_0   \Delta \big)^{-1}  u_j   \\
&\qquad\qquad +\big(\Id- \tau^{\alpha}  b_0   \Delta \big)^{-1}  \big(W(\cdot,t_n)-W(\cdot,t_{n-1}) \big)  \\
=& \big(\Id- \tau^{\alpha}  b_0   \Delta \big)^{-1} 
u_{n-1}+  \tau^{\alpha} \sum_{j=0}^{n-1} b_{n-j}\Delta\big(\Id- \tau^{\alpha}  b_0   \Delta \big)^{-1}  u_j \\
&\qquad\qquad+\sum_{j=1}^\infty \big(W_j(t_n)-W_j(t_{n-1}) \big)  \big(1+ \tau^{\alpha}  b_0   \lambda_j \big)^{-1}  \phi_j ,
\end{aligned} 
\end{equation} 
where $\Id$ denotes the identity operator. 

The main result of this paper is the following theorem. 

\begin{theorem}\label{MainTHM}
Let $\alpha\in(0,2/d)$ with $d\in\{1,2,3\}$. Then, for each $n=1,2,\dots,N,$ the numerical solution $u_n$ given by \eqref{CQ-scheme2} is well defined in $L^2(\Omega;L^2(\mathcal{O}))$ and converges to the mild solution $u(\cdot,t_n)$ with sharp order of convergence, i.e., we have 
\begin{align}\label{main-estimate}
\max_{1\le n\le N} \, \bigg({\mathbb E}\,\|u(\cdot,t_n)-u_n\|_{L^2(\mathcal{O})}^2\bigg)^{\frac12}
\le C\tau^{\frac12-\frac{\alpha d}{4}} ,
\end{align}
where ${\mathbb E}$ denotes the expectation operator and the constant $C$ is independent of $T$. 
\end{theorem}

\section{Proof of Theorem \ref{MainTHM}}
\label{Proof of Theorem}
\subsection{The numerical solution in $L^2(\Omega;L^2(\mathcal{O}))$}

In this subsection, we show that the numerical solution is well defined in $L^2(\Omega;L^2(\mathcal{O}))$. To this end, we use the following estimate for the eigenvalues of the Laplacian operator.
For the simplicity of notations, we denote by $(\cdot,\cdot)$ and $\|\cdot\|$ the inner product and norm of $L^2(\mathcal{O})$, respectively. \medskip
 
\begin{lemma}[\cite{Laptev,LY}]\label{ineq-eigen} 
Let $\mathcal{O}$ denote a bounded domain in $\mathbb{R}^d$, $d\in\{1,2,3\}$. Suppose $\lambda_j$ denotes the $j^{\rm th}$ eigenvalue of the Dirichlet boundary problem for the Laplacian operator $-\Delta$ in $\mathcal{O}$. With $|\mathcal{O}|$ denoting the volume of $\mathcal{O}$, we have that 
\begin{align}
\lambda_j\ge \frac{C_d d}{d+2} j^{2/d} |\mathcal{O}|^{-2/d} 
\end{align}
for all $j\ge 1$, where $C_d=(2\pi)^2 B_d^{-2/d}$ and $B_d$ denotes the volume of the unit $d$-dimensional ball.  
\end{lemma}\medskip

\begin{lemma}\label{2-order-M} 
Under the assumptions of Theorem \ref{MainTHM}, the numerical solution given by \eqref{CQ-scheme-} is well defined in $L^2(\Omega;L^2(\mathcal{O}))$. 
\end{lemma} 

\noindent{\it Proof.}$\,\,$  
Clearly, if $u_j\in L^2(\Omega;L^2(\mathcal{O}))$ for $j=0,\dots,n-1$, then 
\begin{align}\label{numer-sol-1}
\big(\Id- \tau^{\alpha}  b_0   \Delta \big)^{-1} 
u_{n-1}+  \tau^{\alpha} \sum_{j=0}^{n-1} b_{n-j}\Delta\big(\Id- \tau^{\alpha}  b_0   \Delta \big)^{-1}  u_j\in L^2(\Omega;L^2(\mathcal{O})) .
\end{align} 
In view of \eqref{CQ-scheme-}, we only need to prove 
\begin{align}\label{numer-sol-def}
\sum_{j=1}^\infty \big(W_j(t_n)-W_j(t_{n-1}) \big)  \big(1+ \tau^{\alpha}  b_0   \lambda_j \big)^{-1}  \phi_j 
\in L^2(\Omega;L^2(\mathcal{O})) .
\end{align} 
In fact, we have 
\begin{align*}
&{\mathbb E}\bigg\|\sum_{j=\ell}^{\ell+m} 
\big(W_j(t_n)-W_j(t_{n-1}) \big)  \big(1+ \tau^{\alpha}  b_0 \lambda_j \big)^{-1}  \phi_j \bigg\|^2\\
&={\mathbb E} \sum_{j=\ell}^{\ell+m} 
|W_j(t_n)-W_j(t_{n-1}) |^2 \big(1+ \tau^{\alpha}  b_0   \lambda_j \big)^{-2} 
= \sum_{j=\ell}^{\ell+m} 
\tau \big(1+ \tau^{\alpha}  b_0   \lambda _j\big)^{-2} \\
&\le  Cb_0^{-2}\tau^{1-2\alpha} \sum_{j=\ell}^{\ell+m}    j^{-4/d}  \rightarrow 0\quad\mbox{as}\,\,\,\ell\rightarrow\infty .
\end{align*} 
Hence, for a fixed time step $\tau$, 
$$\sum_{j=1}^\ell \big(W_j(t_n)-W_j(t_{n-1}) \big)  \big(1+ \tau^{\alpha}  b_0   \lambda_j \big)^{-1}  \phi_j ,
\quad 
\ell=1,2,\dots$$ is a Cauchy sequence in $L^2(\Omega;L^2(\mathcal{O}))$. 
Consequently, \eqref{numer-sol-def} is proved.  In view of \eqref{CQ-scheme-} and \eqref{numer-sol-1}-\eqref{numer-sol-def}, 
the numerical solution $u_n$ is well defined in $L^2(\Omega;L^2(\mathcal{O}))$. 
\qed \medskip

\subsection{A technical lemma}

To prove the error estimate in Theorem \ref{MainTHM}, we need the following technical lemma. \medskip

\begin{lemma}\label{ineq-sum}
Let $\alpha\ge 0$ and $d\in\{1,2,3\}$. Then there exist constants $C$ and $C_\varphi$ such that  
\begin{align}
&
\sum_{j=1}^\infty \bigg(\frac{r^{\alpha}}{r^\alpha  + \lambda_j }\bigg)^2 
\le Cr^{\alpha d/2} \quad \forall\, r>0, \label{ineq-sum2} \\
&\bigg|\frac{1}{z  + \lambda_j }\bigg|
\le \frac{C_\varphi  }{|z|   + \lambda_j } 
\qquad \qquad
\forall\, z  \in \Sigma_{\varphi}\,\,\,\mbox{with}\,\,\, \varphi\in(0,\pi) , \label{ineq-sum-}
\end{align}
where the constant $C$ depends on the dimension $d$ and the volume of the domain $\mathcal{O}$, and $C_\varphi$ depends on the angle $\varphi\in(0,\pi)$.  
\end{lemma}
\begin{proof}
Clearly, Lemma \ref{ineq-eigen} implies $\lambda_j \ge C  j^{2/d}$, with a constant $C$ depending on the dimension $d$ and the volume of the domain $\mathcal{O}$. 

First, if $0<r< 1$, 
\begin{align} 
&
\sum_{j=1}^\infty \bigg(\frac{r^{\alpha}}{r^\alpha  + \lambda_j }\bigg)^2 
\le 
\sum_{j=1}^\infty  \frac{r^{2\alpha}}{Cj^{4/d} }  
\le Cr^{2\alpha}
\le Cr^{\alpha d/2} ,\quad\mbox{($d/2\le 2$ for $d=1,2,3$)}.
\end{align}

Second, if $r\ge 1$, by setting $M=\lfloor r^{\alpha d/2} \rfloor \ge 1$ to be the largest integer that does not exceed $r^{\alpha d/2}$, we have 
\begin{align*}
\sum_{j=1}^\infty \bigg(\frac{r^{\alpha}}{r^\alpha  + \lambda_j }\bigg)^2
&\le \sum_{j=1}^\infty \bigg(\frac{r^{\alpha}}{r^\alpha  + Cj^{2/d} }\bigg)^2 \\
&= \sum_{j=1}^{M+1} \bigg(\frac{r^{\alpha}}{r^\alpha  + Cj^{2/d} }\bigg)^2
  +\sum_{j=M+2}^{\infty} \bigg(\frac{r^{\alpha}}{r^\alpha  + Cj^{2/d} }\bigg)^2 
=:I_1+I_2.    
\end{align*}
It is easy to see that $I_1\le M+1\le 2M\le 2r^{\alpha d/2}$ and 
\begin{align*}
I_2
&\le \int_{r^{\alpha d/2}}^\infty \bigg(\frac{r^\alpha}{r^\alpha+Cs^{2/d}} \bigg)^2 \d s 
= C^{-d/2}r^{\alpha d/2}\int_{C^{d/2}}^\infty \bigg(\frac{1}{1+\xi^{2/d}} \bigg)^2 \d\xi 
\le Cr^{\alpha d/2},
\end{align*} 
where the equality follows by changing the variable $s=C^{-d/2}r^{\alpha d/2}\xi$. This proves \eqref{ineq-sum2} in the case $r\ge 1$.

Finally, for the point $\xi=-\lambda_j+0\ii$ in the complex plane, we have $|\xi|=\lambda_j$ and $|z-\xi|=|z+\lambda_j|$. 
By looking at the triangle with three vertices $z$, $0$, and $\xi$ with interior angles $\omega_z$, $\omega_0$, and $\omega_{\xi}$ at the three vertices, respectively, we have 
$$
\frac{|z-\xi|}{\sin(\omega_0)}
=\frac{|z|}{\sin(\omega_{\xi})}
=\frac{\lambda_j}{\sin(\omega_z)} .
$$
If $\omega_0\ge \pi/2$, then $|z-\xi|$ 
would be the length of the longest side of the triangle, i.e., 
$$
|z-\xi|\ge |z|\quad\mbox{and}\quad
|z-\xi|\ge \lambda_j 
$$
which immediately implies $$
|z-\xi|\ge \frac{1}{2}(|z|+ \lambda_j) .
$$
If $\omega_0\le \pi/2$, then the angle condition $|{\rm arg}(z)|<\varphi$ implies $\omega_0>\pi-\varphi$. Hence, we have 
$$
|z-\xi|=\frac{|z|\sin(\omega_0)}{\sin(\omega_{\xi})}
\ge |z| \sin(\varphi)
\quad
\mbox{and}
\quad
|z-\xi|=\frac{\lambda_j\sin(\omega_0)}{\sin(\omega_{z})}
\ge \lambda_j \sin(\varphi)
$$
which immediately implies $$
|z-\xi|\ge \frac{\sin(\varphi)}{2}(|z|+ \lambda_j) .
$$
In either case, we have \eqref{ineq-sum-}. 
This completes the proof of Lemma \ref{ineq-sum}. 
\end{proof}

\subsection{Solution representations} 

In this subsection, we derive a representation of the semidiscrete solution $u_n$ by means of the discrete analogue of the Laplace transform and generating function. 

Let $\Gamma_{\theta,\kappa}^{(\tau)}$ denote the truncated piece of the contour $\Gamma_{\theta,\kappa}$ defined by 
\begin{align}\label{trunc-contour}
\Gamma_{\theta,\kappa}^{(\tau)}:=
\{z\in\Gamma_{\theta,\kappa}\,\,:\,\,|\textrm{Im}(z)|\leq \pi/\tau\} .
\end{align}
For $\rho\in(0,1)$, let $\Gamma^{(\tau)}_\rho$ denote the segment of a vertical line defined by 
\begin{align}\label{contour-rho}
\Gamma^{(\tau)}_\rho:=
\{z=-\ln (\rho)/\tau+\ii y\,\,:\,\,y\in\mathbb{R} \textrm{ and } |y|\leq \pi/\tau \}.
\end{align}
The following technical lemma will be used in this and next subsections, where ${\rm arccot}(\cdot)$ denotes the inverse of the cotangent function ${\rm cot}: (0,\pi)\rightarrow\mathbb{R}$. 

\begin{lemma} \label{ineq-1} 
Let $\alpha\ge 0$ and $\theta\in\big(\frac{\pi}{2},{\rm arccot}(-\frac{2}{\pi})\big)$ be given, and let $\rho\in(0,1)$ be fixed, with $\delta_\tau(\zeta)$ defined in \eqref{definition-d}. Then, both $\delta(e^{-z\tau})$ and $(\delta(e^{-z\tau})^\alpha-\Delta)^{-1}$ are analytic with respect to $z$ in the region enclosed by 
$$
\mbox{$\Gamma^{(\tau)}_\rho$,\,\,\, $\Gamma^{(\tau)}_{\theta,\kappa}$,\,\,\, and the two lines 
$\mathbb{R}\pm\ii\pi/\tau$}
\quad\mbox{whenever}\quad 0<\kappa\le \min(1/T,-\ln(\rho)/\tau).
$$ 
Furthermore, we have the following estimates: 
\begin{align}
&\delta_\tau(e^{-z\tau})\in\Sigma_{\theta}&\forall\, z\in\Gamma_{\theta,\kappa}^{(\tau)} 
\label{angle-delta} \\
&C_0|z|\le |\delta_\tau(e^{-\tau z})|\le C_1|z| 
&\forall\, z\in\Gamma_{\theta,\kappa}^{(\tau)}  \label{z-delta}  \\
&|\delta_\tau(e^{-\tau z})-z|\le C\tau|z|^2   
&\forall\, z\in\Gamma_{\theta,\kappa}^{(\tau)}  \\
&|\delta_\tau(e^{-\tau z})^\alpha-z^\alpha|\le C\tau|z|^{\alpha+1} 
&\forall\, z\in\Gamma_{\theta,\kappa}^{(\tau)},  \label{zalpha-delta} 
\end{align}
where the constants $C_0$, $C_1$, and $C$ are independent of $\tau$ and $\kappa\in(0,\min(\frac{1}{T},-\frac{\ln(\rho)}{\tau}))$. 
\end{lemma}
\begin{proof}
Clearly, \eqref{angle-delta} is a consequence of the following two inequalities: 
\begin{align}
&0\le {\rm arg}\bigg(\frac{1-e^{-z\tau}}{\tau}\bigg)\le {\rm arg}(z)
&&\mbox{if}\,\,\, 0\le {\rm arg}(z)\le \theta ,  \label{argz-0} \\
&-{\rm arg}(z)\le {\rm arg}\bigg(\frac{1-e^{-z\tau}}{\tau}\bigg)\le 0
&&\mbox{if}\,\,\, -\theta\le {\rm arg}(z)\le 0 , \label{argz-0-}
\end{align}
which can be proved in the following way when $\frac{\pi}{2}\le \theta\le {\rm arccot}\big(-\frac{2}{\pi}\big)$. 

If ${\rm arg}(z)=\varphi\in[0,\theta] $ and $0\le {\rm Im}(z)\le \pi/\tau$ (thus $0\le \tau |z|\sin(\varphi)\le \pi$), then it is easy to see that ${\rm arg}\big(\frac{1-e^{-\tau z}}{\tau}\big)\ge 0$ and 
\begin{align*}
\cot \bigg({\rm arg}\bigg(\frac{1-e^{-\tau z}}{\tau}\bigg)\bigg)
&=\frac{1-e^{-\tau |z|\cos(\varphi)}\cos(\tau |z|\sin(\varphi)) }{e^{-\tau |z|\cos(\varphi)}\sin (\tau |z|\sin(\varphi)) } \\
&=\frac{e^{\tau |z|\cos(\varphi) }-\cos(\tau |z|\sin(\varphi)) }{\sin (\tau |z|\sin(\varphi)) } \\
&\ge \frac{1+\tau |z|\cos(\varphi) -\cos(\tau |z|\sin(\varphi)) }{\sin (\tau |z|\sin(\varphi)) } 
\ \quad\mbox{(Taylor's expansion)} \\
&= \frac{1+\omega\cot(\varphi)  -\cos(\omega) }{\sin (\omega) } 
\ \ \quad\quad \mbox{(set $\omega=\tau |z|\sin(\varphi)\in[0,\pi)$)} .
\end{align*}
We shall prove $\cot\big({\rm arg}\big(\frac{1-e^{-\tau z}}{\tau}\big)\big)\ge \cot(\varphi)$ so that 
$0\le {\rm arg}\big(\frac{1-e^{-\tau z}}{\tau}\big)\le \varphi={\rm arg}(z)$. 
To this end, we consider the function 
$$
g(\omega):=1+\omega\cot(\varphi)  -\cos(\omega) -\sin (\omega)\cot(\varphi) ,\quad\omega\in[0,\pi] , 
$$
whose derivative is 
$$
g'(\omega):=\frac{\cos(\varphi)  -\cos(\omega+\varphi)}{\sin(\varphi)} . 
$$
In the case $\varphi\in(0,\frac{\pi}{2}]$, $g'(\omega)\ge 0$ for $\omega\in[0,\pi]$. In the case $\varphi\in(\frac{\pi}{2},\pi)$, $g'(\omega)\ge 0$ for $\omega\in[0,2\pi-2\varphi]$ and $g'(\omega)\le 0$ for $\omega\in[2\pi-2\varphi,\pi]$.  
In either case, the function $g(\omega)$ achieves its minimum value at one of the two end points $\omega=0$ and $\omega=\pi$, with 
$$
g(0)=0\quad\mbox{and}\quad g(\pi)=2+\pi\cot(\varphi). 
$$
If $\frac{\pi}{2}\le \theta\le {\rm arccot}\big(-\frac{2}{\pi}\big)$, we then have $g(\pi)\ge 0$. Consequently, 
$g(\omega)\ge 0$ for all $\omega\in[0,\pi]$ and $\cot\big({\rm arg}\big(\frac{1-e^{-\tau z}}{\tau}\big)\big)\ge \cot(\varphi)$ which implies 
$$
0\le {\rm arg}\bigg(\frac{1-e^{-\tau z}}{\tau}\bigg)\le \varphi={\rm arg}(z).
$$
This proves \eqref{argz-0}. The inequality \eqref{argz-0-} can be proved in the same way. 
This completes the proof of \eqref{angle-delta} which further implies that 
$\delta(e^{-z\tau})$ and $(\delta(e^{-z\tau})^\alpha-\Delta)^{-1}$ are analytic with respect to $z$  
in the region enclosed by 
$$
\mbox{$\Gamma^{(\tau)}_\rho$,\,\,\, $\Gamma^{(\tau)}_{\theta,\kappa}$\,\,\, and the two lines 
$\mathbb{R}\pm\ii\pi/\tau$}
\quad\mbox{whenever}\quad 0<\kappa\le -\ln(\rho)/\tau.
$$ 

The estimates \eqref{z-delta}-\eqref{zalpha-delta} are simple consequences of Taylor's theorem. 
\end{proof}

\begin{remark}
The condition $\kappa\le \frac{1}{T}$ is not needed in the proof of this lemma, but is needed in the estimates of the next subsection such as \eqref{esti-f}. 
\end{remark}

To derive the representation of the numerical solution $u_n$, we introduce some notations. 
Let $\bar\partial_\tau W$ be defined by 
\begin{align}\label{pw-const}
&\bar\partial_\tau W(\cdot,t_0):=0 \\
&\bar\partial_\tau W(\cdot,t):=\frac{W(\cdot,t_n)-W(\cdot,t_{n-1})}{\tau}&&\mbox{for}\,\,\, t\in(t_{n-1},t_n],\,\,\, n=1,2,\dots,N\\
&\bar\partial_\tau W(\cdot,t):=0  &&\mbox{for } t> t_{N} ,
\end{align}
where we have set $\bar\partial_\tau W(\cdot,t)=0 $ for $t>t_N$; this does not affect the value of $u_n$, $n=1,2,\dots,N$, upon solving \eqref{CQ-scheme2}.   
Similarly, we define 
\begin{align}\label{pw-const-}
&\bar\partial_\tau W_j(t_0):=0 \\
&\bar\partial_\tau W_j(t):=\frac{W_j(t_n)-W_j(t_{n-1})}{\tau}&&\mbox{for}\,\,\, t\in(t_{n-1},t_n],\,\,\, n=1,2,\dots,N\\
&\bar\partial_\tau W_j(t):=0  && \mbox{for}\,\,\,t> t_{N} .
\end{align}
With these definitions, there are only a finite number of nonzero terms in the sequence 
$\bar\partial_\tau W(\cdot,t_n)$, $n=0,1,2,\dots$. Consequently, the generating function
$$
\widetilde{\bar\partial_\tau W}(\cdot,\zeta)=\sum_{n=0}^\infty \bar\partial_\tau W(\cdot,t_n)\zeta^n
$$
is well defined (polynomial in $\zeta$). Then, we have the following result. 

\begin{proposition}\label{un-rep} 
For the time-stepping scheme \eqref{CQ-scheme2}, the semidiscrete solution $u_n$ can be represented by 
\begin{align}\label{u-semi}
u_n  
&=
\int_0^{t_n}  
E_\tau(t_n-s)\bar\partial_\tau W(\cdot,s)\d s \\
&=
 \sum_{j=1}^\infty \int_{0}^{t_n} E_\tau(t_n-s) \phi_j \bar\partial_\tau W_j(s) \d s  ,
\end{align} 
where the operator $E_\tau(\cdot)$ is given by 
\begin{equation}\label{eqn:EF2}
E_\tau(t) \phi:=\frac{1}{2\pi {\rm i}}\int_{\Gamma_{\theta,\kappa}^{(\tau)}}e^{zt} \frac{z\tau }{e^{z\tau}-1}\delta_\tau(e^{-z\tau})^{\alpha-1}(\delta_\tau(e^{-z\tau})^\alpha  -\Delta )^{-1}\phi\, \d z  \quad
\forall\, \phi\in L^2(\mathcal{O})
\end{equation}
with integration over the truncated contour $\Gamma_{\theta,\kappa}^{(\tau)}$ defined in \eqref{trunc-contour}, oriented with increasing imaginary parts, with the parameters $\kappa$ and  $\theta$ satisfying the conditions of Lemma \ref{ineq-1}.  
\end{proposition}
\begin{proof}
In view of definition \eqref{gener-func} and the identity \eqref{generate-dv}, 
multiplying \eqref{CQ-scheme2} by $\zeta^n$ and summing up the results over $n=0,1,2,\dots$ yield 
\begin{align}
\delta_\tau(\zeta)\widetilde u(\zeta)-\delta_\tau(\zeta)^{1-\alpha}\Delta\widetilde u(\zeta)
=\widetilde{\bar\partial_\tau W}(\cdot,\zeta).
\end{align} 
Then, 
\begin{align}\label{tilde-u}
&\widetilde u(\zeta)  = \delta_\tau(\zeta)^{\alpha-1}(\delta_\tau(\zeta)^\alpha  -\Delta )^{-1}\widetilde{\bar\partial_\tau W}(\cdot,\zeta).
\end{align} 
The function $\widetilde u(\zeta)$ defined in \eqref{tilde-u} is analytic with respect to $\zeta$ in a neighborhood of the origin. By Cauchy's integral formula, it implies that for $\rho\in(0,1)$  
\begin{align*}
u_n
=
\frac{1}{2\pi\ii}\int_{|\zeta|=\rho} \zeta^{-n-1}\widetilde u(\zeta) \d\zeta
=
\frac{\tau}{2\pi\ii}\int_{\Gamma^{(\tau)}_\rho} e^{zt_n}\widetilde u (e^{-z\tau})\d z,
\end{align*}
where the second equality is obtained by the change of variables $\zeta=e^{-z\tau}$, with the contour $\Gamma^{(\tau)}_\rho$ defined in \eqref{contour-rho}. 

From Lemma \ref{ineq-1}, we see that both $\delta(e^{-z\tau})$ and 
$(\delta(e^{-z\tau})^\alpha-\Delta)^{-1}$ are analytic with respect to $z$ in the region $\Sigma\subset\mathbb{C}$ enclosed by $\Gamma^{(\tau)}_\rho$, $\Gamma^{(\tau)}_{\theta,\kappa}$, and the two lines $\mathbb{R}\pm\ii\pi/\tau$. Thus, $e^{zt_n}\widetilde u(e^{-z\tau})$ is analytic with respect to $z\in\Sigma$. Because the values of $e^{zt_n}\widetilde u(e^{-z\tau})$ on the two lines $\mathbb{R}\pm\ii\pi/\tau$ coincide, it follows that (by applying Cauchy's integral formula) 
\begin{align}\label{un-rep-gammatau}
u_n
&=
\frac{\tau}{2\pi\ii}\int_{\Gamma^{(\tau)}_\rho} e^{zt_n}\widetilde u (e^{-z\tau})\d z \nonumber\\
&=
\frac{\tau}{2\pi\ii}\int_{\Gamma^{(\tau)}_{\theta,\kappa}} e^{zt_n}\widetilde u (e^{-z\tau})\d z
+\frac{\tau}{2\pi\ii}\int_{\mathbb{R}+\frac{\ii\pi}{\tau}} e^{zt_n}\widetilde u (e^{-z\tau})\d z \nonumber \\
&\qquad\qquad\qquad\qquad\qquad\quad\,\,\,
-\frac{\tau}{2\pi\ii}\int_{\mathbb{R}-\frac{\ii\pi}{\tau}} e^{zt_n}\widetilde u (e^{-z\tau})\d z \nonumber \\
&=
\frac{\tau}{2\pi\ii}\int_{\Gamma^{(\tau)}_{\theta,\kappa}} e^{zt_n}\widetilde u (e^{-z\tau})\d z \nonumber \\
&=
\frac{\tau}{2\pi i}\int_{\Gamma_{\theta,\kappa}^{(\tau)}} e^{zt_n}\delta_\tau(e^{-z\tau})^{\alpha-1}(\delta_\tau(e^{-z\tau})^\alpha  -\Delta )^{-1}
\widetilde{\bar\partial_\tau W}(\cdot,e^{-z\tau}) \d z 
\nonumber \\
&=
\frac{\tau}{2\pi i}\int_{\Gamma_{\theta,\kappa}^{(\tau)}} e^{zt_n}\delta_\tau(e^{-z\tau})^{\alpha-1}(\delta_\tau(e^{-z\tau})^\alpha  -\Delta )^{-1} 
\frac{z}{e^{z\tau}-1} \widehat {\bar\partial_\tau W}(\cdot,z) \d z , 
\end{align}
where we have substituted \eqref{tilde-u} into the above equality and used the following (straightforward to check) identity in the last step:
$$
\widetilde{\partial_\tau W}(\cdot,e^{-z\tau})
=\frac{z}{e^{z\tau}-1} \widehat {\bar\partial_\tau W}(\cdot,z
) 
$$
with $\widehat {\bar\partial_\tau W}$ denoting the Laplace transform (in time) of the piecewise constant function $\bar\partial_\tau W$.

Through the Laplace transform rule 
\begin{align}\label{LT-rule}
{\mathcal L}^{-1}(\widehat f\, \widehat g)(t)
=\int_0^t{\mathcal L}^{-1}(\widehat f \, )(t-s){\mathcal L}^{-1}(\widehat g)(s)\d s , 
\end{align}
one can derive \eqref{u-semi} from \eqref{un-rep-gammatau}.
The proof of Proposition \ref{un-rep} is complete. 
\end{proof}

\subsection{Error estimate}
In this subsection, we derive an error estimate for the numerical scheme \eqref{CQ-scheme2}. The following lemma is concerned with the difference between the kernels of \eqref{eqn:EF} and \eqref{eqn:EF2}. It will be used in the proof of Theorem \ref{MainTHM}. 

\begin{lemma}\label{ineq-f} 
Let $\alpha\in(0,2/d)$ be given and let $\delta_\tau(\zeta)$ be defined as in 
\eqref{definition-d} with the parameters $\kappa$ and $\theta$ satisfying the conditions of Lemma \ref{ineq-1}. Then, we have  
\begin{align*}
&\bigg|z^{\alpha-1}(z^\alpha+\lambda_j)^{-1} 
-\frac{z\tau }{e^{z\tau}-1} \delta_\tau(e^{-\tau z})^{\alpha-1}(\delta_\tau(e^{-\tau z})^\alpha+\lambda_j)^{-1}\bigg|  
 \le
 \frac{C\tau|z|^\alpha }{|z|^\alpha+\lambda_j }  ,
\ \forall\, z\in \Gamma_{\theta,\kappa}^{(\tau)}. 
\end{align*}
\end{lemma}
\begin{proof}
By the triangle inequality and Lemma \ref{ineq-1}, we have 
\begin{align*}
&|z^{\alpha-1}(z^\alpha+\lambda_j)^{-1} 
- \frac{z\tau }{e^{z\tau}-1} \delta_\tau(e^{-\tau z})^{\alpha-1}(\delta_\tau(e^{-\tau z})^\alpha+\lambda_j)^{-1}| \\
&\quad\le
\bigg| \frac{e^{z\tau}-1-z\tau }{e^{z\tau}-1} \bigg|
|z^{\alpha-1}(z^\alpha+\lambda_j)^{-1}| \\
&\qquad +
\bigg|\frac{z\tau }{e^{z\tau}-1}\bigg|
|z|^{\alpha-1}|(z^\alpha+\lambda_j)^{-1}-(\delta_\tau(e^{-\tau z})^\alpha+\lambda_j)^{-1} | \\
&\qquad 
+ 
\bigg|\frac{z\tau }{e^{z\tau}-1}\bigg||z^{\alpha-1}-\delta_\tau(e^{-\tau z})^{\alpha-1}| |(\delta_\tau(e^{-\tau z})^\alpha+\lambda_j)^{-1}| =: \mathcal{J}_1 + \mathcal{J}_2 + \mathcal{J}_3 ,
\end{align*}
where 
\begin{align*}
\mathcal{J}_1 
&\le C|z\tau| |z^{\alpha-1}(z^\alpha+\lambda_j)^{-1}| 
\quad\mbox{\big(using the Taylor expansion of $\big| \frac{e^{z\tau}-1-z\tau }{e^{z\tau}-1} \big|$\big)} \\
\mathcal{J}_2 
&\le C
|z|^{\alpha-1}|z^\alpha-\delta_\tau(e^{-\tau z})^\alpha| 
|(z^\alpha+\lambda_j)^{-1}(\delta_\tau(e^{-\tau z})^\alpha+\lambda_j)^{-1}|  \\
&\le C \tau|z|^{2\alpha}  
|(z^\alpha+\lambda_j)^{-1}(\delta_\tau(e^{-\tau z})^\alpha+\lambda_j)^{-1}|
\qquad\,\,\,\mbox{(here we use \eqref{zalpha-delta})} \\
&\le C\tau|z|^{2\alpha} 
(|z|^\alpha+\lambda_j)^{-1}(|\delta_\tau(e^{-\tau z})|^\alpha+\lambda_j)^{-1} \\
&\le
\frac{C\tau|z|^\alpha}{|z|^\alpha+\lambda_j}.
\end{align*}
In the estimates above we have used the following inequality (cf. \cite[inequality C.1]{GLW-2017})
\begin{align}\label{ztau-eztau}
\bigg|\frac{z\tau}{e^{z\tau}-1}\bigg|\le C,\quad\forall\, z\in \Gamma_{\theta,\kappa}^{(\tau)}.
\end{align}

The last inequality is due to Lemma \ref{ineq-sum} together with the angle condition ${\rm arg}(z^\alpha)\le \alpha\theta<\pi$ and ${\rm arg}(\delta_\tau(e^{-\tau z})^\alpha)\le \alpha\theta<\pi$ (cf. Lemma \ref{ineq-1}). 
Furthermore, we have 
\begin{align*} 
\mathcal{J}_3 
&\le C|z^{\alpha-1}-\delta_\tau(e^{-\tau z})^{\alpha-1}| |(\delta_\tau(e^{-\tau z})^\alpha+\lambda_j)^{-1}| \\
&\le C\Big(|z^{\alpha}-\delta_\tau(e^{-\tau z})^{\alpha}| |z|^{-1}
+ |z^{-1}-\delta_\tau(e^{-\tau z})^{-1}||\delta_\tau(e^{-\tau z})|^{\alpha} \Big)|(\delta_\tau(e^{-\tau z})^\alpha+\lambda_j)^{-1}| \\ 
&\le  \Big(C\tau |z|^{1+\alpha} |z|^{-1}  
+C \tau |z|^2 |z|^{-1}  |\delta_\tau(e^{-\tau z})|^{\alpha-1} \Big)|(\delta_\tau(e^{-\tau z})^\alpha+\lambda_j)^{-1}| \\ 
&
\le  
\frac{C\tau|z|^\alpha}{|z|^\alpha+\lambda_j} 
\qquad\qquad \mbox{(here we use Lemma \ref{ineq-sum} and Lemma \ref{ineq-1})} .
\end{align*}
The proof of Lemma \ref{ineq-f} is complete.  
\end{proof}

Now, we start to prove Theorem \ref{MainTHM}. 
From \eqref{Mild-sol} and \eqref{eqn:EF} we see that the mild solution admits the decomposition 
\begin{align} \label{u-repr-}
u(\cdot,t)
&= \sum_{j=1}^\infty  \phi_j\int_0^t F^{(\tau)}_j (t-s) \d W_j(s)  
+ \sum_{j=1}^\infty \phi_j\int_0^t H^{(\tau)}_j (t-s) \d W_j(s)  
\end{align}
with 
\begin{align} 
&F^{(\tau)}_j (t)   :=  \frac{1}{2\pi {\rm i}}\int_{\Gamma_{\theta,\kappa}^{(\tau)}}e^{zt} z^{\alpha-1} (z^\alpha +\lambda_j)^{-1} \, \d z  \\
&H^{(\tau)}_j(t)  :=  \frac{1}{2\pi i}\int_{\Gamma_{\theta,\kappa}\backslash\Gamma_{\theta,\kappa}^{(\tau)}} e^{zt} z^{\alpha-1}(z^\alpha  +\lambda_j )^{-1}  \d z  . \label{def-Hjt}
\end{align} 
Also, \eqref{u-semi} and \eqref{eqn:EF2} imply 
\begin{align} \label{un-repr-}
u_n  
&=\sum_{j=1}^\infty\phi_j \int_0^{t_n} E^{(\tau)}_j(t_n-s) \bar\partial_\tau W_j(s) \d s 
\end{align} 
with  
\begin{equation} \label{expr-Etau}
E^{(\tau)}_j(t):=\frac{1}{2\pi {\rm i}}\int_{\Gamma_{\theta,\kappa}^{(\tau)}}e^{zt} \frac{z\tau }{e^{z\tau}-1}\delta_\tau(e^{-z\tau})^{\alpha-1}(\delta_\tau(e^{-z\tau})^\alpha +\lambda_j )^{-1} \, \d z . 
\end{equation}
Comparing \eqref{u-repr-} and \eqref{un-repr-} yields  
\begin{align} \label{error}
u(\cdot,t_n)-u_n  
&= \sum_{j=1}^\infty  \phi_j\int_0^{t_n} (F^{(\tau)}_j (t_n-s) -E^{(\tau)}_j (t-s))\d W_j(s)  \nonumber\\
&\quad 
+\sum_{j=1}^\infty  \phi_j\int_0^{t_n} E^{(\tau)}_j (t_n-s) \Big (\d W_j(s) -\bar\partial_\tau W_j(s)\d s \Big) \nonumber\\
&\quad 
+ \sum_{j=1}^\infty \phi_j\int_0^{t_n} H^{(\tau)}_j (t_n-s) \d W_j(s)   \nonumber\\
&=: {\mathcal E}_\tau(t_n) + {\mathcal G}_\tau(t_n)+ {\mathcal H}_\tau(t_n) .
\end{align} 
Then Theorem \ref{MainTHM} is a consequence of the following lemma. 
The proof of Theorem \ref{MainTHM} is complete. \qed

\begin{lemma}\label{E-G-H}
Under the assumptions of Theorem \ref{MainTHM}, we have 
${\mathcal E}_\tau(t_n) , {\mathcal G}_\tau(t_n), {\mathcal H}_\tau(t_n)\in L^2(\Omega;L^2(\mathcal{O}))$, satisfying the following estimate: 
\begin{align}
{\mathbb E} \|{\mathcal E}_\tau(t_n)\|^2
+{\mathbb E} \|{\mathcal G}_\tau(t_n)\|^2
+{\mathbb E} \|{\mathcal H}_\tau(t_n)\|^2 \leq C\tau^{1-\alpha d/2}. 
\end{align}
\end{lemma}
\noindent{\it Proof.}$\,\,$  
First, we estimate ${\mathcal H}_\tau(t_n)$. By choosing a number $\beta\in(\alpha d/2,1)$ and using Lemma \ref{ineq-sum}, 
we have 
\begin{align}\label{E-Htau}
{\mathbb E}\, \|{\mathcal H}_\tau(t_n) \|^2
&= \int_0^{t_n}  \sum_{j=1}^\infty |H_j^{(\tau)}(t_n-s) |^2 \d s 
\quad\mbox{(It\^o's isometry)}\nonumber \\
&
=\int_0^{t_n}  \sum_{j=1}^\infty  |H_j^{(\tau)}(s) |^2 \d s \nonumber \\
&= \int_0^{t_n} \sum_{j=1}^\infty \bigg|\frac{1}{2\pi i}\int_{\Gamma_{\theta,\kappa}\backslash\Gamma_{\theta,\kappa}^{(\tau)}} e^{zs} z^{\alpha-1}(z^\alpha  + \lambda_j )^{-1}  \d z\bigg|^2 \d s \nonumber\\
&\le \int_0^{t_n}  \sum_{j=1}^\infty \bigg|\frac{1}{2\pi i}\int_{\Gamma_{\theta,\kappa}\backslash\Gamma_{\theta,\kappa}^{(\tau)}}  \frac{z^{\alpha}}{ z^\alpha  + \lambda_j } \frac{e^{zs}}{z}  \d z\bigg|^2 \d s 
\nonumber\\
&\le  C \int_0^{t_n} \sum_{j=1}^\infty \bigg(\int_{\Gamma_{\theta,\kappa}\backslash\Gamma_{\theta,\kappa}^{(\tau)}} \frac{|\d z|}{ |z|^{2-\beta} }  \bigg)\bigg(  \int_{\Gamma_{\theta,\kappa}\backslash\Gamma_{\theta,\kappa}^{(\tau)}}  \bigg|\frac{z^{\alpha}}{ z^\alpha  + \lambda_j }\bigg|^2  \frac{ |e^{zs}|^2 }{ |z|^{\beta} }  |\d z|\bigg) \d s \nonumber\\
&\le   C \int_0^{t_n} \sum_{j=1}^\infty\bigg(\int_{1/\tau}^\infty   \frac{\d r}{ r^{2-\beta}  } \bigg)
 \bigg(\int_{1/\tau}^\infty 
\bigg|\frac{r^{\alpha}}{r^\alpha  + \lambda_j }\bigg|^2    \frac{e^{(2s  \cos\theta) r} }{ r^{\beta} }  \d r  \bigg)\d s  
\nonumber\\ 
&\le C \tau^{1-\beta} \int_0^{t_n}  \int_{1/\tau}^\infty \sum_{j=1}^\infty \bigg(\frac{r^{\alpha}}{r^\alpha  + \lambda_j }\bigg)^2  \frac{e^{(2s  \cos\theta) r} }{ r^{\beta} }   \d r   \d s 
\nonumber\\ 
&\le C \tau^{1-\beta} \int_0^{t_n}  \int_{1/\tau}^\infty  r^{\alpha d/2-\beta} e^{(2s \cos\theta) r}   \d r   \d s \nonumber\\ 
&\le C \tau^{1-\beta}    \int_{1/\tau}^\infty  r^{\alpha d/2-\beta-1} (1-e^{(2t_n \cos\theta) r})   \d r    \nonumber\\ 
&\le C \tau^{1-\beta}   \tau^{\beta-d\alpha/2}  \nonumber\\
&\le C \tau^{1-\alpha d/2} .
\end{align} 

Next, we estimate ${\mathcal E}_\tau(t_n)$. To this end, we apply Lemma \ref{ineq-f} and obtain 
\begin{align}\label{ftau}
&|F^{(\tau)}_j (s) -E^{(\tau)}_j (s)|^2 \nonumber \\
&= 
\biggl|\frac{1}{2\pi \ii}\int_{\Gamma_{\theta,\kappa}^{(\tau)}} e^{zs} 
\bigg(\frac{z^{\alpha-1}}{z^\alpha+ \lambda_j } 
- \frac{z\tau }{e^{z\tau}-1} \frac{\delta_\tau(e^{-\tau z})^{\alpha-1}}{\delta_\tau(e^{-\tau z})^\alpha  + \lambda_j  } \bigg) \d z \bigg|^2 \nonumber\\
&\le 
\bigg(\int_{\Gamma_{\theta,\kappa}^{(\tau)}} |e^{zs}| 
\frac{C\tau|z|^\alpha }{|z|^\alpha+\lambda_j } |\d z|  \bigg)^2
 \nonumber \\ 
&\le 
C\tau^2 
\bigg( \int_{\Gamma_{\theta,\kappa}^{(\tau)}} |\d z| \bigg) 
\int_{\Gamma_{\theta,\kappa}^{(\tau)}} 
\bigg(\frac{|z|^\alpha }{|z|^\alpha+\lambda_j } \bigg)^2 |e^{zs}|^2 |\d z|   \nonumber \\ 
&\le 
C\tau 
\int_{\Gamma_{\theta,\kappa}^{(\tau)}} 
\bigg(\frac{|z|^\alpha }{|z|^\alpha+\lambda_j } \bigg)^2 |e^{zs}|^2 |\d z|  .  
\end{align}
By using the expression of $\mathcal{E}_\tau(t_n)$,  
we have 
\begin{align}\label{E-Etau}
{\mathbb E}\, \|{\mathcal E}_\tau(t_n) \|^2
&= \sum_{j=1}^\infty  {\mathbb E}\,\bigg|\int_0^{t_n} (F^{(\tau)}_j (t_n-s) -E^{(\tau)}_j(t_n-s))\d W_j(s)\bigg|^2  \nonumber\\
&= \sum_{j=1}^\infty  \int_0^{t_n} \big|F^{(\tau)}_j(t_n-s)-E^{(\tau)}_j(t_n-s)\big|^2 \d s  \nonumber\\
&= \sum_{j=1}^\infty  \int_0^{t_n} \big|F^{(\tau)}_j(s)-E^{(\tau)}_j(s)\big|^2 \d s  \nonumber\\
&\le 
C\tau 
\int_0^{t_n}\int_{\Gamma_{\theta,\kappa}^{(\tau)}} 
\sum_{j=1}^\infty  \bigg(\frac{|z|^\alpha }{|z|^\alpha+\lambda_j } \bigg)^2 |e^{zs}|^2 |\d z|  \d s \nonumber \\
&\le
C\tau\int_0^{t_n}\int_{\Gamma_{\theta,\kappa}^{(\tau)}} |z|^{\alpha d/2}|e^{zs}|^2|\d z| \d s, 
\end{align} 
where the last inequality follows from Lemma \ref{ineq-sum}. 
Since $t_n\ge \tau$ and 
\begin{align*}
\Gamma_{\theta,\kappa}^{(\tau)}
&=\{z\in\mathbb{C}:z=r e^{\pm\ii\theta},r\ge\kappa, r |\sin(\theta)| \le\pi/\tau \} 
\cup \{z\in\mathbb{C}:|z|=\kappa,|\arg z|\le\theta\} ,
\end{align*}
by choosing $\kappa \le \frac{2}{t_n|\sin(\theta)|}$, we have 
\begin{align}\label{esti-f}
&{\mathbb E} \|\mathcal{E}_\tau(t_n)\|^2 \nonumber \\
&\le
C\tau\int_0^{t_n}\int_\kappa^{\frac{\pi}{\tau|\sin(\theta)|}} r^{\alpha d/2} e^{(2s\cos\theta)r}\d r\d s 
\nonumber\\
&\quad 
+C\tau\int_0^{t_n}\int_{-\theta}^\theta 
\kappa^{\alpha d/2+1} e^{2s\cos(\psi)\kappa}  \d \psi \d s
\nonumber\\
&\le
C\tau\int_\kappa^{\frac{\pi}{\tau|\sin(\theta)|}} r^{\alpha d/2-1}(1-e^{(2t_n\cos\theta)r})\d r 
+C\tau\int_0^{t_n} \int_{-\theta}^\theta \kappa^{\alpha d/2+1} e^{2s\kappa} \d\psi\d s
\nonumber\\
&\le
C\tau^{1-\alpha d/2}+C\tau \kappa^{\alpha d/2}(e^{2t_n\kappa}-1) \nonumber\\
&\leq
C\tau^{1-\alpha d/2}.  
\end{align} 

Finally, we estimate $\mathcal{G}_\tau(t_n)$. Because $\bar \partial_\tau W_j(t_n)=\frac{1}{\tau}\int_{t_{n-1}}^{t_n}\d W_j(s)$, we obtain
\begin{align}
\mathcal{G}_\tau(t_n) 
&=\sum_{j=1}^\infty  \phi_j\int_0^{t_n} E^{(\tau)}_j (t_n-s) \Big (\d W_j(s) -\bar\partial_\tau W_j(s)\d s \Big)  \nonumber \\ 
&=
\sum_{j=1}^\infty  \phi_j\sum_{i=1}^n\bigg(\int_{t_{i-1}}^{t_i}E^{(\tau)}_j(t_n-s)\d W_j(s)-\int_{t_{i-1}}^{t_i}E^{(\tau)}_j(t_n-\xi)\bar \partial_\tau W_j(t_i)\d \xi\bigg) \nonumber\\
&=
\sum_{j=1}^\infty  \phi_j\sum_{i=1}^n\int_{t_{i-1}}^{t_i} \bigg(\frac{1}{\tau}\int_{t_{i-1}}^{t_i}(E^{(\tau)}_j(t_n-s)-E^{(\tau)}_j(t_n-\xi))\d \xi\bigg)\d W_j(s) .
\end{align}
Then, 
\begin{align} \label{Estimate-I2}
{\mathbb E}\|\mathcal{G}_\tau(t_n) \|^2
&=
\sum_{j=1}^\infty {\mathbb E}\bigg|\sum_{i=1}^n\int_{t_{i-1}}^{t_i}\bigg(\frac{1}{\tau}\int_{t_{i-1}}^{t_i} \Big(E^{(\tau)}_j(t_n-s)-E^{(\tau)}_j(t_n-\xi)\Big) \d\xi\bigg) \d W_j(s)  \bigg|^2 \nonumber\\
&=
\sum_{j=1}^\infty  \sum_{i=1}^n\int_{t_{i-1}}^{t_i} \bigg|\frac{1}{\tau}\int_{t_{i-1}}^{t_i} \Big(E^{(\tau)}_j(t_n-s)-E^{(\tau)}_j(t_n-\xi)\Big) \d\xi \bigg|^2 \d s .
\end{align}
By using the expression \eqref{expr-Etau}, for $|s-\xi|\le \tau$ we have 
\begin{align}
&\Big|E^{(\tau)}_j(t_n-s)-E^{(\tau)}_j(t_n-\xi)\Big|^2 \nonumber\\
&=
\bigg|\frac{1}{2\pi\ii}\int_{\Gamma_{\theta,\kappa}^{(\tau)}} e^{z(t_n-s)}(1-e^{z(s-\xi)})\delta_\tau(e^{-z\tau})^{\alpha-1}(\delta_\tau(e^{-z\tau})^\alpha+\lambda_j)^{-1}\frac{z\tau}{e^{z\tau}-1}\d z \bigg|^2 \nonumber\\
&\le
C\bigg(\int_{\Gamma_{\theta,\kappa}^{(\tau)}}|\d z| \bigg)\bigg(\int_{\Gamma_{\theta,\kappa}^{(\tau)}}|e^{z(t_n-s)}|^2 |1-e^{z(s-\xi)}|^2\bigg|\frac{\delta_\tau(e^{-z\tau})^{\alpha-1}}{\delta_\tau(e^{-z\tau})^\alpha+\lambda_j}\bigg|^2\bigg|\frac{z\tau}{e^{z\tau}-1}\bigg| |\d z| \bigg) \nonumber \\
&\le
C\tau^{-1}\int_{\Gamma_{\theta,\kappa}^{(\tau)}} |e^{z(t_n-s)}|^2 \tau^2|z|^2 \bigg|\frac{\delta_\tau(e^{-z\tau})^{\alpha-1}}{\delta_\tau(e^{-z\tau})^\alpha+\lambda_j}\bigg|^2 |\d z| \nonumber \\
&\le
C\tau \int_{\Gamma_{\theta,\kappa}^{(\tau)}} |e^{z(t_n-s)}|^2 |z|^2 \bigg(\frac{|z|^{\alpha-1}}{|z|^\alpha+\lambda_j} \bigg)^2 |\d z| 
\quad\mbox{(here we use Lemma \ref{ineq-sum})}, \nonumber \\
\end{align}
where we have used $ |1-e^{z(s-\xi)}|\le C|z(s-\xi)|\le C|z|\tau$ and \eqref{ztau-eztau} in deriving the second to last inequality. 
Substituting the last inequality into \eqref{Estimate-I2} yields 
 \begin{align}\label{E-Gtau}
{\mathbb E}\|\mathcal{G}_\tau(t_n) \|^2 
&\le
\sum_{j=1}^\infty C\tau\int_0^{t_n} \int_{\Gamma_{\theta,\kappa}^{(\tau)}} |e^{z(t_n-s)}|^2 |z|^2 \bigg(\frac{|z|^{\alpha-1}}{|z|^\alpha+\lambda_j} \bigg)^2 |\d z| \d s \nonumber\\
&\le
C\tau\int_0^{t_n}  \int_{\Gamma_{\theta,\kappa}^{(\tau)}} |e^{z(t_n-s)}|^2 |z|^2 \sum_{j=1}^\infty \bigg(\frac{|z|^{\alpha-1}}{|z|^\alpha+\lambda_j} \bigg)^2 |\d z| \d s \nonumber\\
&\le
C\tau\int_0^{t_n} \int_{\Gamma_{\theta,\kappa}^{(\tau)}} |e^{z(t_n-s)}|^2 |z|^{\alpha d/2} |\d z| \d s \nonumber\\
&=
C\tau\int_0^{t_n} \int_{\Gamma_{\theta,\kappa}^{(\tau)}} |e^{zs}|^2 |z|^{\alpha d/2} |\d z| \d s
\quad\mbox{(here we use a change of variable)} \nonumber\\
&\le
C\tau^{1-\alpha d/2},
\end{align}
where the last inequality can be estimated in the same way as \eqref{E-Etau}. 
\qed

\section{Numerical examples}\label{Numerical examples}
\setcounter{equation}{0}
In this section, we present three numerical examples to illustrate the theoretical analyses. 

\vspace{0.1in}
{\bf Example 1.} We first consider the one-dimensional stochastic time-fractional equation
\begin{align}\label{eq-alpha}
\partial_t u(x,t)-\partial^2_{x}\partial_t^{1-\alpha} u(x,t)= f(x,t)+\varepsilon\dot W(x,t)  \end{align}
for $0\le x\le 1$, $0<t\le 1$, with homogenous Dirichlet boundary condition and $0$ initial condition. In the above equation, 
\begin{align*}
f(x,t)=2tx^2(1-x)^2-\frac{2t^{1+\alpha}}{\Gamma(2+\alpha)}(2-12x+12x^2),
\end{align*}
$\varepsilon$ is a given constant, and $W$ the cylindrical Wiener process. 
In the absence of white noise, the exact solution would be $u_d(x,t)=t^2x^2(1-x)^2$, which corresponds to the exact mean of the stochastic solution. 

We discretize the problem \eqref{eq-alpha} in time by using the scheme \eqref{CQ-scheme2} and, in space, by continuous piecewise linear finite element method. 
Here, $h=1/M$ denotes the spatial mesh size and $U^n(x)$ the numerical solution of the fully discrete scheme. We take $\tau=h=2^{-5}$ and $\varepsilon=0.1$. For each computation, $I=1000$ independent realizations are performed with different Wiener processes. For each realization $\omega_i$, $i=1,\dots,I$, we generate $M$ independent Brownian motions $W_j(t)$, $j=1,\dots,M$. In Figure 4.1(left), we present the exact solution $u_d$ of the deterministic problem, the mean  value of numerical solutions for \eqref{eq-alpha}, and the standard deviation, respectively, at $t_n=1$. 
Moreover, the numerical approximations $U^n(x,\omega_i)$, $i=1,2,3$ of $u(x,t_n,\omega_i)$, with three independent realizations, are given in Figure 4.1(right) at $t_n=1$. The numerical simulations in Figure 4.1 are performed by taking $\alpha=0.5$. Similar results are shown in Figure 4.2 for $\alpha=1.3$. 
Because the solution has $C^{\min(\frac{1}{\alpha}-\frac12,1)}(\overline\Omega)$ pathwise regularity in space (cf. \cite[Proposition 2]{MijenaNane}), the numerical solution for $\alpha=0.5$ (the solution is $C^{1}(\overline\Omega)$) is smoother than the numerical solution for $\alpha=1.3$ (the solution is $C^{0.27}(\overline\Omega)$); see  Figures \ref{fig41} and \ref{fig42} for a visual comparison.

\begin{figure}[htp]
\begin{center}
\begin{tabular}{c}
\includegraphics[width=2.5in]{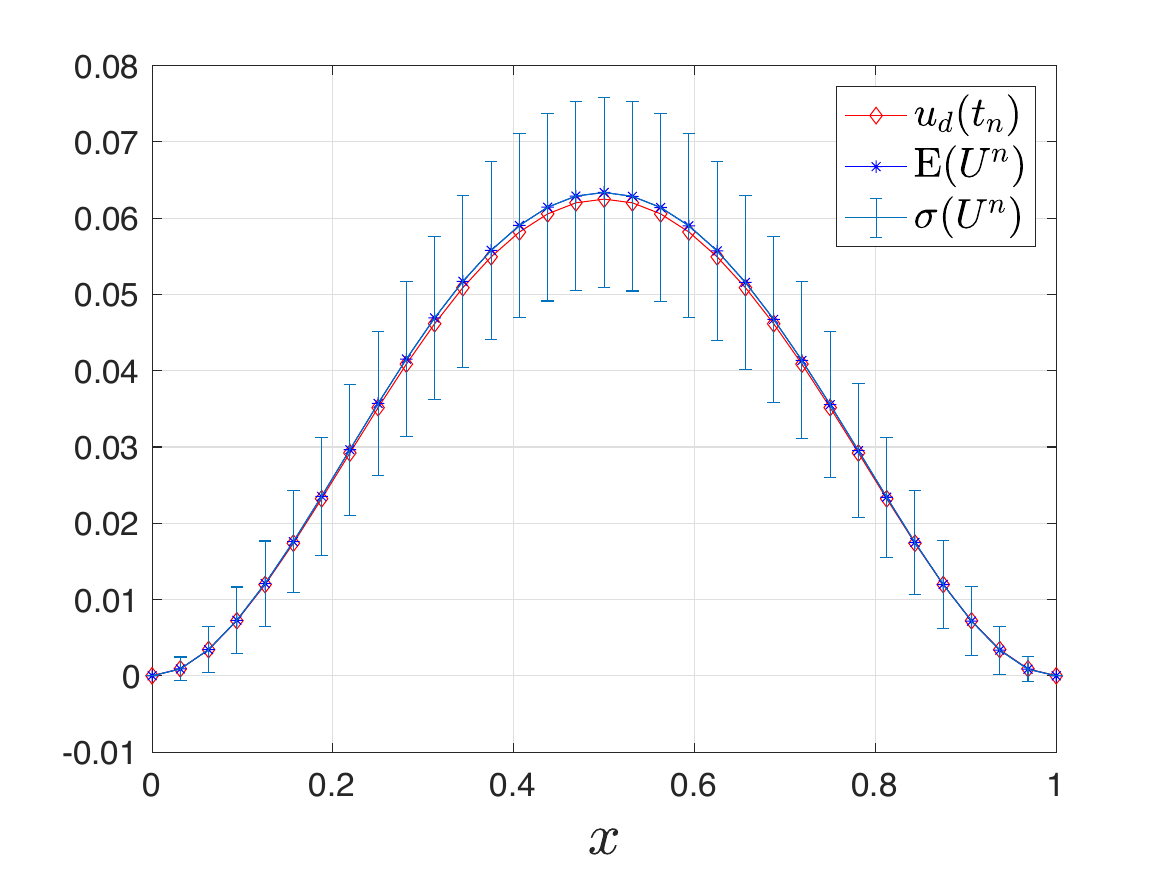}
\includegraphics[width=2.5in]{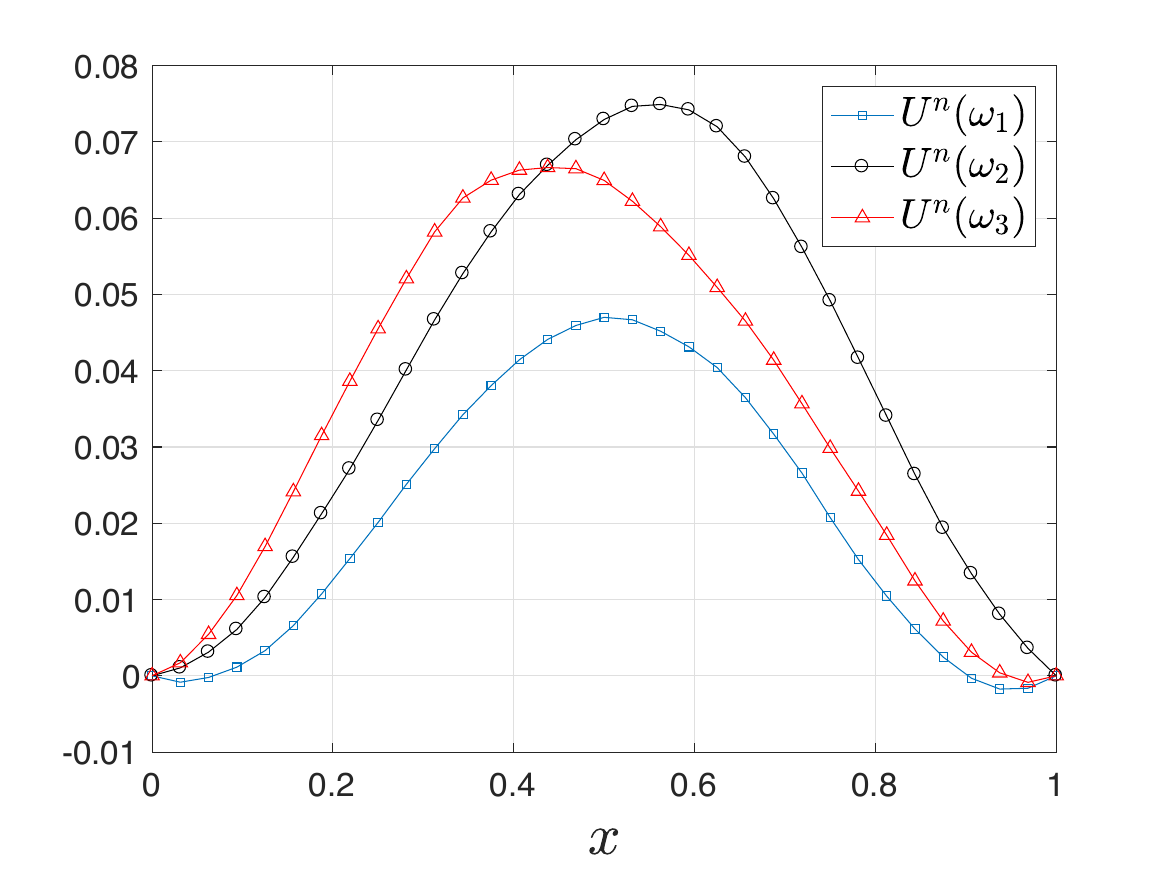}
\end{tabular}
\end{center}
\caption{Numerical approximations for $u(x,t)$ with $\alpha=0.5$}\label{fig41}
\end{figure}

\begin{figure}[htp]
\begin{center}
\begin{tabular}{c}
\includegraphics[width=2.5in]{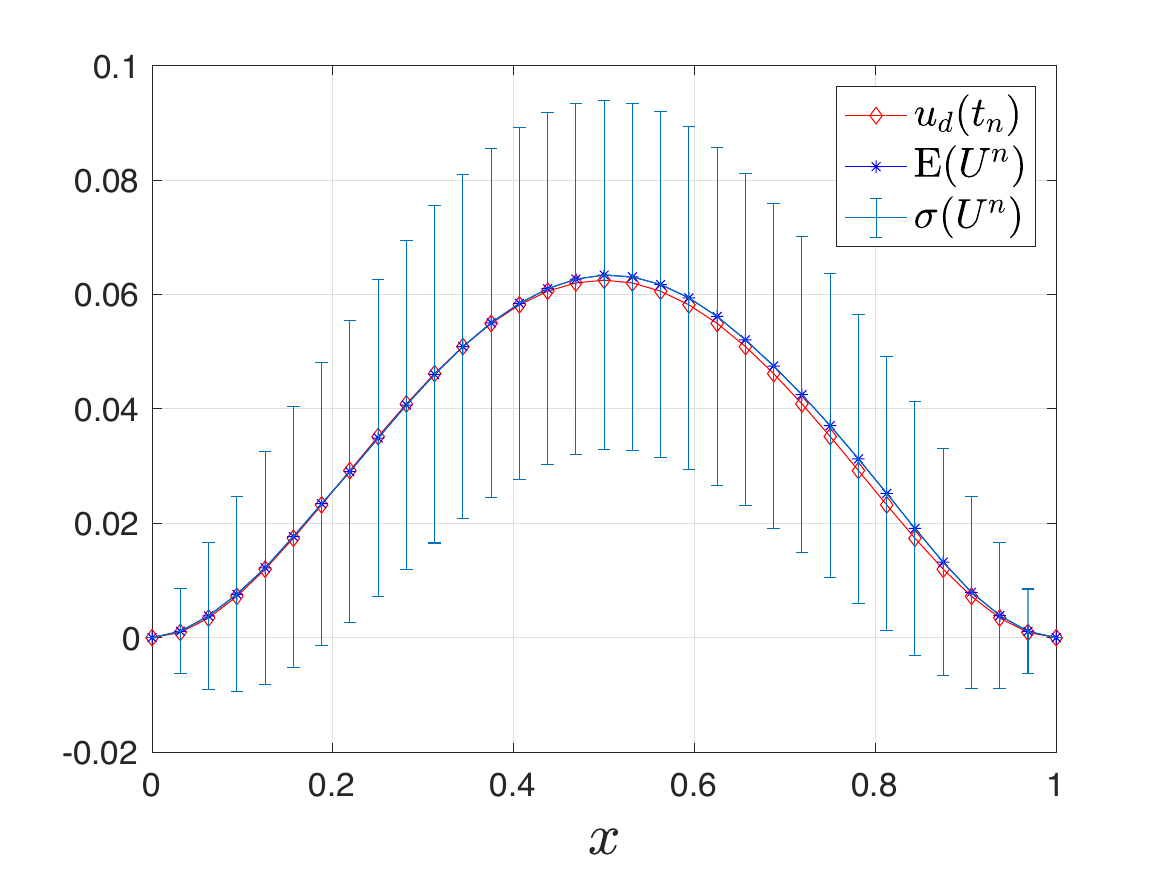}
\includegraphics[width=2.5in]{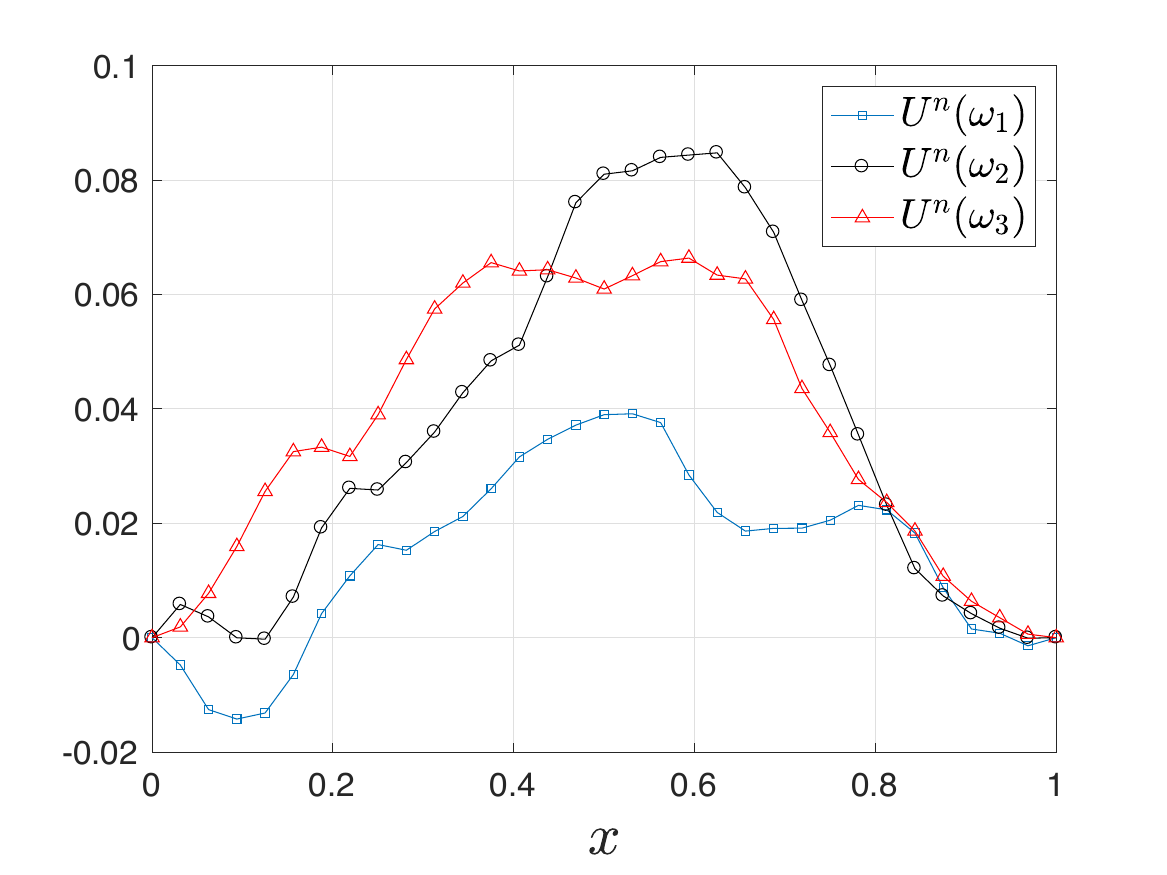}
\end{tabular}
\end{center}
\caption{Numerical approximations for $u(x,t)$ with $\alpha=1.3$}\label{fig42}
\end{figure}

\vspace{0.1in}

{\bf Example 2.} We next consider the convergence rate of the numerical scheme \eqref{CQ-scheme2} for \eqref{eq-alpha} with $\varepsilon=1$. The problem \eqref{eq-alpha} is discretized using backward-Euler convolution quadrature and a linear Galerkin finite element method. To investigate the convergence rate, we consider $I=1000$ independent realizations for each time step $\tau_k=2^{-k}$, $k=5,\dots,8$. 
In order to focus on the time discretization error, we solve the time-discrete stochastic PDE \eqref{CQ-scheme2} using a sufficiently small spatial mesh size $h=1/M=2^{-9}$ so that the spatial discretization error is relatively negligible. 
Then the error $E(\tau_k)$ is computed by  
\begin{align}\label{t-order}
E(\tau_k)=\bigg(\frac{1}{I}\sum_{i=1}^I \|U^{N,\tau_k}(\cdot,\omega_i)-U^{N,\tau_{k-1}}(\cdot,\omega_i)\|^2\bigg)^{\frac{1}{2}}
\end{align}for $k=6,7,8$.

In \cite{LubichSloanThomee:1996}, it is proved that  the backward-Euler convolution quadrature for time-fractional PDE \eqref{Deter-SPDE2} is first-order convergent. 
Thus, by Theorem \ref{MainTHM}, the convergence order of the scheme \eqref{CQ-scheme2} for problem \eqref{eq-alpha} should be $O(\tau^{\frac{1}{2}-\frac{\alpha}{4}})$ in a one-dimensional spatial domain. Consequently, we expect the error $E(\tau_k)$ to have the convergence rate
\begin{align}\label{time-order}
\log_2\frac{E(\tau_k)}{E(\tau_{k+1})} 
\approx 
\log_2\bigg(\frac{\tau_k}{\tau_{k+1}}\bigg)^{\frac{1}{2}-\frac{\alpha}{4}}=\frac{1}{2}-\frac{\alpha}{4}
\end{align}
for successive halvings of the time step. We test the above result by taking $\alpha=0.5$ and $0.9$ for a subdiffusion setting and $\alpha=1.3$ and $1.7$ for a diffusion-wave setting. From \eqref{time-order}, $\log_2\frac{E(\tau_k)}{E(\tau_{k+1})}\approx 0.375$ for $\alpha=0.5$,  
$\log_2\frac{E(\tau_k)}{E(\tau_{k+1})}\approx 0.275$ for $\alpha=0.9$,   $\log_2\frac{E(\tau_k)}{E(\tau_{k+1})} \approx 0.175$ for $\alpha=1.3$, 
and $\log_2\frac{E(\tau_k)}{E(\tau_{k+1})} \approx 0.075$ for $\alpha=1.7$. 
Clearly, the results in Table \ref{time-error} ($t_n=1$) illustrate the sharp convergence rate.

\begin{table}[!ht]
\begin{center}
\caption{$E(\tau_k)$ and convergence rates in 1D.}
\label{time-error}
\begin{tabular}{c|ccc|c}
\hline
\hline
\ {$\alpha\backslash \tau_k$}&$2^{-6}$&$2^{-7}$&$2^{-8}$&order\  \\
\hline
\ $\alpha=0.5$ &1.075e-02&8.284e-03&6.382e-03&0.376 (0.375)\ \\
\ $\alpha=0.9$ &2.825e-02&2.340e-02&1.921e-02&0.278 (0.275)\ \\
\ $\alpha=1.3$ &6.340e-02&5.654e-02&5.004e-02&0.171 (0.175)\ \\
\ $\alpha=1.7$ &1.415e-01&1.352e-01&1.275e-01&0.075 (0.075)\ \\
\hline
\hline
\end{tabular} 
\end{center}
\end{table}

\vspace{0.1in}
{\bf Example 3.} Lastly, we consider the stochastic time-fractional equation 
\begin{align}\label{eq-alpha-2}
\partial_t u(x,t)-\Delta\partial_t^{1-\alpha} u(x,t)= f(x,t)+\dot W(x,t)  
\end{align}
in the two-dimensional spatial domain $[0,1]\times[0,1]$, with homogenous Dirichlet boundary condition and $0$ initial condition. Here, we choose $0<t\le 1$ and 
\begin{align*}
f(x,t)
&=2tx_1^2x_2^2(1-x_1)^2(1-x_2)^2
-\frac{2t^{1+\alpha}}{\Gamma(2+\alpha)}(2-12x_1+12x_1^2)x_2^2(1-x_2)^2 \\
&\quad -\frac{2t^{1+\alpha}}{\Gamma(2+\alpha)}x_1^2(1-x_1)^2(2-12x_2+12x_2^2)
\end{align*}
for $x=(x_1,x_2)\in[0,1]\times[0,1]$. 
The exact solution of the corresponding deterministic problem is $u_d(x,t)=t^2 x_1^2x_2^2(1-x_1)^2(1-x_2)^2$.

We solve the stochastic equation \eqref{eq-alpha-2} using the backward-Euler scheme \eqref{CQ-scheme2}, where spatial discretization is effected by the standard piecewise linear Galerkin finite element method. 
A uniform triangular partition with 50 nodes in each direction is used. 
Similarly, we choose $\tau_k=2^{-k}$, $k=4,\dots,7$, and consider $I=1000$ independent realizations to investigate the temporal convergence rate. 
For each realization $\omega_i$, $i=1,\dots,I$, we generate $2500$ independent Brownian motions $W_j(t)$, $j=1,\dots,2500$.
The mesh size $h=\frac{\sqrt{2}}{50}$ is fixed so that spatial error is relatively negligible. 
Then, the error \eqref{t-order} is computed for each fixed time step $\tau_k$ and presented in Table \ref{time-error-2} at $t_n=1$. By Theorem \ref{MainTHM}, the convergence rate of the scheme \eqref{CQ-scheme2} for problem \eqref{eq-alpha-2} is $O(\tau^{\frac{1}{2}-\frac{\alpha}{2}})$ in the two spatial dimensional setting. 
Clearly, the numerical results are consistent with the theoretical analyses given in 
Theorem \ref{MainTHM}.

\begin{table}[!ht]
\begin{center}
\caption{$E(\tau_k)$ and convergence rates in 2D.}
\label{time-error-2}
\begin{tabular}{c|ccc|c}
\hline
\hline
\ {$\alpha\backslash \tau_k$}&$2^{-5}$&$2^{-6}$&$2^{-7}$&order\  \\
\hline
\ $\alpha=0.3$ &3.848e-03&2.941e-03&2.335e-03&0.361 (0.35)\ \\
\ $\alpha=0.5$ &9.992e-03&8.556e-03&7.395e-03&0.217 (0.25)\ \\
\ $\alpha=0.7$ &2.115e-02&1.936e-02&1.762e-02&0.132 (0.15)\ \\
\ $\alpha=0.9$ &3.997e-02&3.914e-02&3.851e-02&0.033 (0.05)\ \\
\hline
\hline
\end{tabular} 
\end{center}
\end{table}

\section{Conclusion}
We considered the stability and convergence of numerical approximations of a stochastic time-fractional  PDE by using the backward Euler convolution quadrature in time. By means of a discrete analogue of the inverse Laplace transform, we derived an integral representation of the numerical solution which was then used to prove the sharp convergence rate of the numerical approximation.

Instead of the contour 
$$\Gamma_{\theta}^{(\tau)}=\{z\in\mathbb{C}:|{\rm arg}(z)|=\theta,\,\,|{\rm Im}(z)|\le \pi/\tau\}$$
used in \cite{LubichSloanThomee:1996}, we have used the contour $\Gamma_{\theta,\kappa}^{(\tau)}$ given in \eqref{trunc-contour} for the analysis in this paper. The contour $\Gamma_{\theta,\kappa}^{(\tau)}$ excludes the origin and thus can be used not only for the Dirichlet Laplacian but also for the Neumann Laplacian (whose spectrum includes the origin). Similarly, for the Dirichlet Laplacian, it is not necessary to show that $\delta_\tau(e^{-z\tau})\in \Sigma_\theta$ as in \eqref{angle-delta}. Using this result, the analyses in this paper can be naturally extended to the Neumann Laplacian.

The paper focuses on semidiscretization in time by convolution quadrature. The main contribution of the paper is to show the possibility of removing the $\epsilon$ term in the previously obtained error estimate \eqref{Estimate-epsilon} and to establish a foundation for the further analysis of spatial discretization considered in \cite{GLW-2017}. Our analysis is based on specific growth properties of the eigenvalues of the Laplacian operator (cf. Lemma \ref{ineq-eigen}), and thus may not be directly extended to more general abstract operators such as semilinear problems and multiplicative space-time white noises. For example, for the multiplicative noise case, the identities \eqref{u-repr-}-\eqref{expr-Etau} do not hold and thus the analysis becomes more complicated. Thus, extension to semilinear problems and multiplicative space-time white noises remains open and certainly should be a subject of future research.

Theorem \ref{MainTHM} can be extended to higher order moments by using the Burkh\"older-Davis-Gundy inequality (cf. \cite[(1.1)]{NNeerven-Veraar-Weis-2007}, \cite[(6.29)]{PratoZabczyk2014}, or \cite{Pratelli-1988})$:$  
for all $p\in(1,\infty)$ there exists $C_p >0$ such that 
$$
\mathbb{E} \bigg(\max_{1\le n\le N} \bigg\|\int_0^{t_n}\phi(s) \,\d W(s)\bigg\|^p \bigg)
\le C_p\, \mathbb{E}\bigg[ \bigg(\int_0^T \|\phi(s)\|_{L_2^0}^2 \,\d s\bigg)^{\frac{p}{2}}\bigg] ,
$$
where $L_2^0$ denotes the space of Hilbert-Schmidt operators on $L^2(\mathcal{O})$. 
Let 
\begin{align} 
&H^{(\tau)}(t)  :=  \frac{1}{2\pi i}\int_{\Gamma_{\theta,\kappa}\backslash\Gamma_{\theta,\kappa}^{(\tau)}} e^{zt} z^{\alpha-1}(z^\alpha  - \Delta )^{-1}  \d z  ,
\end{align} 
which is a Hilbert-Schmidt operator satisfying (see \eqref{def-Hjt})
$$\|H^{(\tau)}(t)\|_{L^0_2}^2=\sum_{j=1}^\infty |H^{(\tau)}_j (t)|^2 . $$
By using the definition of $\mathcal{H}_\tau(t_n)$ in \eqref{error}, we have 
\begin{align}
{\mathbb E}\, (\max_{1\le n\le N} \|{\mathcal H}_\tau(t_n) \|^p)
&=
{\mathbb E}\bigg(\max_{1\le n\le N} \bigg\|\int_0^{t_n} H^{(\tau)}(t_n-s)  \d W(s) \bigg\|^p\bigg) \nonumber \\
&\le 
C_p\, {\mathbb E} \bigg[\bigg(\int_0^T \|H^{(\tau)}(t_n-s) \|_{L^0_2}^2 \d s \bigg)^{\frac{p}{2}}\bigg] 
\quad\mbox{(Burkh\"older inequality)} \nonumber \\
&= 
C_p\, {\mathbb E} \bigg[\bigg(\int_0^T \sum_{j=1}^\infty |H^{(\tau)}_j (t_n-s)|^2 \d s \bigg)^{\frac{p}{2}}\bigg] \nonumber \\
&=C_p\,  \bigg(\int_0^T  \sum_{j=1}^\infty |H_j^{(\tau)}(t_n-s) |^2 \d s \bigg)^{\frac{p}{2}}
\nonumber \\
&\le C_p\,  \tau^{(\frac12-\frac{\alpha d}{4})p} ,
\end{align} 
where the last inequality utilizes the result of \eqref{E-Htau}. 
Similarly, the following estimates can be proved by using \eqref{esti-f} and \eqref{E-Gtau}: 
$$
{\mathbb E} (\max_{1\le n\le N}\|{\mathcal E}_\tau(t_n)\|^p)
\le C_p\,  \tau^{(\frac12-\frac{\alpha d}{4})p}
\quad\mbox{and}\quad
{\mathbb E} (\max_{1\le n\le N}\|{\mathcal G}_\tau(t_n)\|^p)
\le C_p\,  \tau^{(\frac12-\frac{\alpha d}{4})p} . 
$$ 
Substituting these estimates into \eqref{error} yields the $p$-moment estimate: 
\begin{align}\label{main-estimate-p}
 \big({\mathbb E}\,\max_{1\le n\le N}\|u(\cdot,t_n)-u_n\|^p\big)^{\frac1p}
\le C_p\, \tau^{\frac12-\frac{\alpha d}{4}} ,\quad\forall\,p\in(1,\infty) .
\end{align}
The estimate above also implies the existence of a random variable $\mathcal{C}$ having finite moments of any order independent of $\tau$ such that the following pathwise estimate holds: 
\begin{align}\label{main-estimate-path}
\max_{1\le n\le N}\|u(\cdot,t_n)-u_n\|
\le \mathcal{C} \, \tau^{\frac12-\frac{\alpha d}{4}} .
\end{align}

The convergence rate proved in this article is optimal with respect to the regularity of the solution in time. However, whether it is the highest possible convergence rate among all possible numerical methods is unknown. For stochastic ODEs, the convergence rates of some numerical methods (such as Milstein's method) may be higher than the regularity of the solution in time. But Milstein's method does not yield higher convergence rate for the stochastic PDE problem considered in this article. For example, in the case $\alpha=d=1$, Milstein's method (equivalent to Euler Maruyama with additive noise) 
\begin{equation}
u_n= u_{n-1} - \tau \Delta  u_{n-1}  
+ W(\cdot,t_n)-W(\cdot,t_{n-1}) 
\end{equation}
does not converge at all (as it is an explicit scheme, which requires a CFL condition that cannot be satisfied by a semi-discretization in time). Even if we modify the Milstein's method to be an implicit scheme 
\begin{equation}\label{euler-un-w}
u_n= u_{n-1} - \tau \Delta  u_{n}  
+ W(\cdot,t_n)-W(\cdot,t_{n-1}) ,
\end{equation}
the scheme only has strong convergence rate $O(\tau^{1/4})$. This is different from stochastic ODEs, for which the strong convergence rates of Milstein's method and the corresponding implicit scheme are $O(\tau)$ (higher than the temporal regularity of the solution). This difference between stochastic PDEs and stochastic ODEs is due to the fact that $\Delta u$ does not have the same temporal regularity as $u$ when the PDE is driven by a space-time white noise \eqref{white-noise}. In particular, $\Delta u$ is not H\"older continuous in time in the space $L^2(\mathcal{O})$ (but Milstein's method requires H\"older continuity of $\Delta u$ to achieve a higher convergence rate). 
In the case of colored noise (instead of space-time white noise), the implicit Euler scheme \eqref{euler-un-w} can achieve a better convergence rate than the temporal H\"older regularity of the solution; see \cite{Wang-2017}. 



\appendix\label{append}
\section{The mild solution in $C^\gamma([0,T];L^2(\Omega;L^2(\mathcal{O})))$}\label{Append}
In the cases $\alpha<\min(1,2/d)$ and $\alpha>1$ with $d=1$, the mild solution of \eqref{Frac-SPDE2} (with space-time white noise) has been studied in different function spaces under different settings. For example, see \cite{KP,MijenaNane}. For the reader's convenience, in this appendix, we illustrate that the mild solution given by \eqref{Mild-sol-}-\eqref{Mild-sol} is indeed well defined in $C^\gamma([0,T];L^2(\Omega;L^2(\mathcal{O})))$, a result used for the numerical analysis in this paper.

\begin{theorem}
The mild solution defined by \eqref{Mild-sol-}-\eqref{Mild-sol} is in $C^{\gamma}([0,T];L^2(\Omega;L^2(\mathcal{O})))$ for arbitrary $\gamma\in(0,\frac12-\frac{\alpha d}{4})$.
\end{theorem}

\noindent{\it Proof}.$\,\,$ 
In \eqref{Mild-sol}, the formula \eqref{eqn:EF} implies
\begin{align}
\label{each-term}
\int_0^t E(t-s)\phi_j\d W_j(s)  
=
\phi_j  \int_0^t h_j(t-s) \d W_j(s)  
\end{align}
for a deterministic time-independent function $\phi_j\in L^2(\mathcal{O})$ and with the deterministic space-independent function $h_j(\cdot)$ given by 
$$
h_j(t-s) 
=\frac{1}{2\pi {\rm i}}\int_{\Gamma_{\theta,\kappa}}e^{z(t-s)} z^{\alpha-1} 
(z^\alpha +\lambda_j)^{-1} \, \d z .
$$
By the theory of the Ito integral and the identity \eqref{each-term}, each term in \eqref{Mild-sol} is well defined in $C([0,T];L^2(\Omega; L^2(\mathcal{O})))$.  
Because the one-dimensional Wiener processes $W_j(s)$, $j=1,2,\dots$, are independent of each other, it follows that 
\begin{align}\label{Est-Et-sW}
&\sup_{t\in[0,T]} {\mathbb E} \bigg\|\sum_{j=\ell}^{\ell+m} \int_0^t E(t-s)\phi_j\d W_j(s)  \bigg\|^2 \nonumber \\
&=\sup_{t\in[0,T]}  \sum_{j=\ell}^{\ell+m} \int_0^t \| E(t-s)\phi_j\|^2\d s 
=\sup_{t\in[0,T]}  \sum_{j=\ell}^{\ell+m} \int_0^t \| E(s)\phi_j\|^2\d s \nonumber \\
&\le\sum_{j=\ell}^{\ell+m} \int_0^T \| E(s)\phi_j\|^2\d s \nonumber \\
&= \int_0^{T} \sum_{j=\ell}^{\ell+m} \bigg|\frac{1}{2\pi i}\int_{\Gamma_{\theta,\kappa}} e^{zs} z^{\alpha-1}(z^\alpha  + \lambda_j )^{-1}  \d z\bigg|^2 \d s \nonumber \\
&\le
C\int_0^{T} \sum_{j=\ell}^\infty\bigg(\int_{\Gamma_{\theta,\kappa}^\theta} \frac{1}{|z|^\beta}|\d z| \bigg) 
\bigg(\int_{\Gamma_{\theta,\kappa}^\theta}  
\bigg|\frac{z^{\alpha}}{z^\alpha+\lambda_j}\bigg|^2\frac{|e^{2zs}| }{|z|^{2-\beta}} |\d z| \bigg)
\d s \nonumber \\
&\quad
+C\int_0^{T} \sum_{j=\ell}^\infty
\bigg(\int_{\Gamma_{\theta,\kappa}^\kappa} \frac{1}{|z|^\beta}|\d z| \bigg) 
\bigg(\int_{\Gamma_{\theta,\kappa}^\kappa}  
\bigg|\frac{z^{\alpha}}{z^\alpha+\lambda_j}\bigg|^2\frac{|e^{2zs}| }{|z|^{2-\beta}} |\d z| \bigg)
\d s \nonumber \\
&\le 
C\int_0^{T} \sum_{j=\ell}^\infty
\bigg(\int_{\kappa}^\infty \frac{1}{r^\beta}\d r \bigg) 
\bigg(\int_{\kappa}^\infty\bigg|\frac{r^{\alpha}}{r^\alpha+\lambda_j}\bigg|^2\frac{e^{-2rs|\cos(\theta)|}}{r^{2-\beta}} \d r \bigg)
\d s  
\quad\mbox{(use Lemma \ref{ineq-sum})} 
\nonumber \\
&\quad
+C\int_0^{T} \sum_{j=\ell}^\infty
\bigg(\int_{-\theta}^\theta  \frac{1}{\kappa^\beta} \kappa\d\varphi \bigg) 
\bigg(\int_{-\theta}^\theta   
\bigg|\frac{\kappa^{\alpha}}{\kappa^\alpha+\lambda_j}\bigg|^2\frac{e^{2\kappa s\cos(\varphi)}}{\kappa^{2-\beta}}\kappa\d\varphi \bigg)
\d s \nonumber \\
&\le 
C\kappa^{1-\beta} \int_0^{T}\int_{\kappa}^\infty \sum_{j=\ell}^\infty \bigg|\frac{r^{\alpha}}{r^\alpha+\lambda_j}\bigg|^2\frac{e^{-2rs|\cos(\theta)|}}{r^{2-\beta}} \d r  \d s   \nonumber \\
&\quad 
+C\kappa^{1-\beta} \int_0^{T}\int_{-\theta}^\theta  \sum_{j=\ell}^\infty  
\bigg|\frac{\kappa^{\alpha}}{\kappa^\alpha+\lambda_j}\bigg|^2\frac{e^{2\kappa s\cos(\varphi)}}{\kappa^{2-\beta}}\kappa\d\varphi \d s ,\nonumber \\
\end{align}
where $\beta\in(1,2-\alpha d/2)$. 

In view of Lemma \ref{ineq-sum}, 
$\displaystyle\sum_{j=1}^\infty \bigg|\frac{r^{\alpha}}{r^\alpha+\lambda_j}\bigg|^2 \le Cr^{\alpha d/2}$ implies 
$\displaystyle\sum_{j=\ell}^\infty \bigg|\frac{r^{\alpha}}{r^\alpha+\lambda_j}\bigg|^2\rightarrow 0$
as $ \ell\rightarrow \infty$ 
and  
\begin{align}\label{int-k-1}
\int_0^{T}\int_{\kappa}^\infty \sum_{j=\ell}^\infty \bigg|\frac{r^{\alpha}}{r^\alpha+\lambda_j}\bigg|^2\frac{e^{-2rs|\cos(\theta)|}}{r^{2-\beta}} \d r  \d s   
&\le \int_0^{T}\int_{\kappa}^\infty \sum_{j=1}^\infty \bigg|\frac{r^{\alpha}}{r^\alpha+\lambda_j}\bigg|^2\frac{e^{-2rs|\cos(\theta)|}}{r^{2-\beta}} \d r  \d s  \nonumber \\
&\le C\int_0^{T}\int_{\kappa}^\infty  r^{\alpha d/2} \frac{e^{-2rs|\cos(\theta)|}}{r^{2-\beta}} \d r  \d s \nonumber  \\
&\le C\int_{\kappa}^\infty  r^{\alpha d/2+\beta-3} (1-e^{-2r T|\cos(\theta)|})  \d r  \nonumber \\
&\le C\kappa^{\alpha d/2+\beta-2} . 
\end{align}
The Lebesgue dominated convergence theorem implies that 
\begin{align*}
\lim_{\ell\rightarrow \infty}
\kappa^{1-\beta} \int_0^{T}\int_{\kappa}^\infty \sum_{j=\ell}^\infty \bigg|\frac{r^{\alpha}}{r^\alpha+\lambda_j}\bigg|^2\frac{e^{-2rs|\cos(\theta)|}}{r^{2-\beta}} \d r  \d s 
=0  .
\end{align*}

Similarly, $\displaystyle\sum_{j=\ell}^\infty \bigg|\frac{\kappa^{\alpha}}{\kappa^\alpha+\lambda_j}\bigg|^2\le C\kappa^{\alpha d/2}$ implies 
$
\displaystyle\sum_{j=\ell}^\infty \bigg|\frac{\kappa^{\alpha}}{\kappa^\alpha+\lambda_j}\bigg|^2\rightarrow 0$ 
as $\ell\rightarrow \infty ,
$
and  
\begin{align}\label{int-k-2}
\int_0^{T}\int_{-\theta}^\theta  \sum_{j=\ell}^\infty  
\bigg|\frac{\kappa^{\alpha}}{\kappa^\alpha+\lambda_j}\bigg|^2\frac{e^{2\kappa s\cos(\varphi)}}{\kappa^{2-\beta}}\kappa\d\varphi \d s  
&\le  \int_0^{T}\int_{-\theta}^\theta  \sum_{j=1}^\infty  
\bigg|\frac{\kappa^{\alpha}}{\kappa^\alpha+\lambda_j}\bigg|^2
\frac{e^{2\kappa s\cos(\varphi)}}{\kappa^{2-\beta}}\kappa
\d\varphi \d s \nonumber  \\
&\le C\int_0^{T}\int_{-\theta}^\theta  \kappa^{\alpha d/2} 
\frac{e^{2\kappa s}}{\kappa^{2-\beta}} 
\kappa\d\varphi \d s \nonumber \\
&\le C\int_{-\theta}^\theta   \kappa^{\alpha d/2+\beta-2} 
(e^{2\kappa T-1})  
\d \varphi  \nonumber \\
&\le C\kappa^{\alpha d/2+\beta-2} .
\end{align}
Again, the Lebesgue dominated convergence theorem implies that 
\begin{align*}
\lim_{\ell\rightarrow \infty}
\kappa^{1-\beta} \int_0^{T}\int_{-\theta}^\theta  \sum_{j=\ell}^\infty  
\bigg|\frac{\kappa^{\alpha}}{\kappa^\alpha+\lambda_j}\bigg|^2\frac{e^{2\kappa s\cos(\varphi)}}{\kappa^{2-\beta}}\kappa\d\varphi \d s 
=0  .
\end{align*}

Overall, we have  
\begin{align*}
\sup_{t\in[0,T]} {\mathbb E}\bigg\|\sum_{j=\ell}^{\ell+m} \int_0^t E(t-s)\phi_j\d W_j(s)  \bigg\|^2 
\rightarrow 0 \quad\mbox{as}\,\,\, \ell\rightarrow \infty
\end{align*}
which implies that the sequence 
$$
\sum_{j=1}^\ell \int_0^t E(t-s)\phi_j\d W_j(s)  ,\quad \ell=1,2,\dots 
$$
is a Cauchy sequence in $C([0,T];L^2(\Omega; L^2(\mathcal{O})))$. 
Consequently, the sequence converges to a function $u\in C([0,T];L^2(\Omega; L^2(\mathcal{O})))$, which is the mild solution defined by \eqref{Mild-sol}. 

Let $L_2^0$ denote the space of Hilbert-Schmidt operators on $L^2(\mathcal{O})$ (cf. \cite[Appendix C]{PratoZabczyk2014}) with the operator norm 
\begin{align}
\|E(t-s)\|_{L_2^0} =
\bigg(\sum_{j=1}^\infty \|E(t-s)\phi_j\|_{L^2(\mathcal{O})}^2\bigg)^{\frac12} .
\end{align}
The above analysis clearly shows that 
\begin{align}
\int_0^t \|E(t-s)\|_{L_2^0}^2 \d s<\infty .
\end{align}
In view of \cite[Proposition 4.20 and page 99]{PratoZabczyk2014}, the stochastic integral \eqref{Mild-sol-} is well defined in $L^2(\Omega;L^2(\mathcal{O}))$, and \eqref{Mild-sol-} coincides with the series representation \eqref{Mild-sol} (\cite[section 4.2.2]{PratoZabczyk2014}).

Similar to the estimate \eqref{Est-Et-sW}, by considering 
\begin{align*}
&\frac{u(\cdot,t)-u(\cdot,t-h)}{h^\gamma} \\
&=\sum_{j=1}^\infty \int_0^{t-h}  \frac{E(t-s)-E(t-h-s)}{h^\gamma}\phi_j\d W_j(s) 
+\frac{1}{h^\gamma}\sum_{j=1}^\infty \int_{t-h}^t  E(t-s)\phi_j\d W_j(s)
\end{align*}
we have  
\begin{align}\label{Holder-u}
&\sup_{t\in[0,T]}{\mathbb E} \bigg\|\frac{u(\cdot,t)-u(\cdot,t-h)}{h^\gamma}\bigg\|^2 \nonumber \\
&=\sup_{t\in[0,T]} \bigg( \sum_{j=1}^{\infty}  \int_0^{t-h} \bigg\| \frac{E(t-s)-E(t-h-s)}{h^\gamma}\phi_j\bigg\|^2\d s \nonumber \\
&\quad 
+\ \frac{1}{h^{2\gamma}}\sum_{j=1}^\infty \int_{t-h}^t  \|E(t-s)\phi_j\|^2\d s  \bigg) \\
&\le\sum_{j=1}^{\infty} \int_0^T \bigg\| \frac{E(s+h)-E(s)}{h^\gamma}\phi_j\bigg\|^2\d s 
+\sum_{j=1}^\infty \frac{1}{h^{2\gamma}}\int_{0}^h  \|E(s)\phi_j\|^2\d s .\nonumber
\end{align} 
Because $\big|\frac{e^{zh}-1}{h^\gamma}\big|\le C|z|^\gamma$ on the contour $\Gamma_{\theta,\kappa}$ (on which ${\rm Re}(z)\le 0$ when $|z|\ge\kappa$), it follows that 
\begin{align*}
&\sum_{j=1}^{\infty} \int_0^T \bigg\| \frac{E(s+h)-E(s)}{h^\gamma}\phi_j\bigg\|^2\d s  \\
&= \int_0^{T} \sum_{j=1}^{\infty} \bigg|\frac{1}{2\pi i}\int_{\Gamma_{\theta,\kappa}} \frac{e^{zh}-1}{h^\gamma}e^{zs} z^{\alpha-1}(z^\alpha  + \lambda_j )^{-1}  \d z\bigg|^2 \d s \\
&\le
C\int_0^{T} \sum_{j=1}^\infty\bigg(\int_{\Gamma_{\theta,\kappa}^\theta} \frac{1}{|z|^\beta}|\d z| \bigg) 
\bigg(\int_{\Gamma_{\theta,\kappa}^\theta}  
\bigg|\frac{z^{\alpha}}{z^\alpha+\lambda_j}\bigg|^2\frac{|e^{2zs}| }{|z|^{2-\beta-2\gamma}} |\d z| \bigg)
\d s  \\
&\quad
+C\int_0^{T} \sum_{j=1}^\infty
\bigg(\int_{\Gamma_{\theta,\kappa}^\kappa} \frac{1}{|z|^\beta}|\d z| \bigg) 
\bigg(\int_{\Gamma_{\theta,\kappa}^\kappa}  
\bigg|\frac{z^{\alpha}}{z^\alpha+\lambda_j}\bigg|^2\frac{|e^{2zs}| }{|z|^{2-\beta-2\gamma}} |\d z| \bigg)
\d s  \\
&\le 
C\int_{\kappa}^\infty  r^{\alpha d/2+\beta+2\gamma-3} (1-e^{-2r T|\cos(\theta)|})  \d r 
+C\int_{-\theta}^\theta   \kappa^{\alpha d/2+\beta+2\gamma-2} 
(e^{2\kappa T-1}) \\
&\le 
C\kappa^{\alpha d/2+\beta+2\gamma-2},
\end{align*}
where the second last inequality requires $\beta>1$ for the improper integral $\int_{\Gamma_{\theta,\kappa}^\kappa} \frac{1}{|z|^\beta}|\d z|$ to be convergent (then the estimates follow similarly as for \eqref{int-k-1}-\eqref{int-k-2}), and the last inequality requires $\alpha d/2+\beta+2\gamma-3<-1$. This requires $2\gamma<1-\frac{\alpha d}{2}$. 
Also,
\begin{align*}
&\sum_{j=1}^{\infty} \frac{1}{h^{2\gamma}}\int_0^h \| E(s)\phi_j\|^2\d s  \\
&= \frac{1}{h^{2\gamma}}\int_0^{h} \sum_{j=1}^{\infty} \bigg|\frac{1}{2\pi i}\int_{\Gamma_{\theta,\kappa}}  e^{zs} z^{\alpha-1}(z^\alpha  + \lambda_j )^{-1}  \d z\bigg|^2 \d s \\
&\le
\frac{C}{h^{2\gamma}}\int_0^{h} \sum_{j=1}^\infty\bigg(\int_{\Gamma_{\theta,\kappa}^\theta} \frac{1}{|z|^\beta}|\d z| \bigg) 
\bigg(\int_{\Gamma_{\theta,\kappa}^\theta}  
\bigg|\frac{z^{\alpha}}{z^\alpha+\lambda_j}\bigg|^2\frac{|e^{2zs}| }{|z|^{2-\beta}} |\d z| \bigg)
\d s  \\
&\quad
+\frac{C}{h^{2\gamma}}\int_0^{h} \sum_{j=1}^\infty
\bigg(\int_{\Gamma_{\theta,\kappa}^\kappa} \frac{1}{|z|^\beta}|\d z| \bigg) 
\bigg(\int_{\Gamma_{\theta,\kappa}^\kappa}  
\bigg|\frac{z^{\alpha}}{z^\alpha+\lambda_j}\bigg|^2\frac{|e^{2zs}| }{|z|^{2-\beta}} |\d z| \bigg)
\d s  \\
&\le 
C\int_{\kappa}^\infty  r^{\alpha d/2+\beta-3} \frac{1-e^{-2r h|\cos(\theta)|}}{h^{2\gamma}}\d r 
+C\int_{-\theta}^\theta   \kappa^{\alpha d/2+\beta-2} 
\frac{e^{2\kappa h-1}}{h^{2\gamma}} \d\varphi \\
&\le 
C\int_{\kappa}^\infty  r^{\alpha d/2+\beta+2\gamma-3} \d r 
+C\int_{-\theta}^\theta   \kappa^{\alpha d/2+\beta+2\gamma-2} \\ 
&\le 
C\kappa^{\alpha d/2+\beta+2\gamma-2},
\end{align*}
which again requires $\beta>1$ and $2\gamma<1-\frac{\alpha d}{2}$ for the convergence of the improper integrals. 

Substituting the last two results into \eqref{Holder-u} yields that $u\in C^{\gamma}([0,T];L^2(\Omega; L^2(\mathcal{O})))$ for arbitrary $\gamma\in(0,\frac12-\frac{\alpha d}{4})$. 
\qed

\vspace{0.1in}
{\bf Acknowledgements.} 
The authors would like to thank Professor Xiaobing Feng for many valuable suggestions and comments. 
\vspace{0.1in}

{\bf Funding.} 
The research of M. Gunzburger and J. Wang was supported in part by the USA National Science Foundation grant DMS-1315259 and by the USA Air Force Office of Scientific Research grant FA9550-15-1-0001.
The work of B. Li was supported in part by the Hong Kong RGC grant 15300817.

\bibliographystyle{abbrv}

\end{document}